\documentclass[onecolumn]{IEEEtran}
\usepackage{graphicx} % Required for inserting images
\usepackage{geometry}
\usepackage{subcaption}
\usepackage{amsthm}
\usepackage[T1]{fontenc}
\usepackage{amsmath}
\usepackage{graphicx}
\usepackage{amssymb,bbm}
\usepackage{mathtools}
\usepackage{float}
\usepackage[linesnumbered,ruled,vlined,noend]{algorithm2e}
\usepackage{tikz}
\usetikzlibrary{shapes,arrows}
\usepackage{hyperref}

\tikzset{
    block/.style={rectangle, draw, line width=0.5mm, black, text width=8em, text centered,
                  rounded corners, minimum height=2em},
    line/.style={draw, thick, ->},
    line_label/.style={draw, thick}
}%

\usepackage{xcolor}

\usepackage{mathtools}

\newtheorem{lemma}{Lemma}
\newtheorem{problem}{Problem}
\newtheorem{example}{Example}
\newtheorem{proposition}{Proposition}

\newtheorem{definition}{Definition}
\newtheorem{assumption}{Assumption}
\newtheorem{remark}{Remark}

\newcommand{\STM}[2]{\Phi(#2,#1)}
\newcommand{\zeros}[1]{0_{#1}}
\newcommand{\eye}[1]{I_{#1}}

\tikzset{
    block/.style={rectangle, draw, line width=0.5mm, black, text width=8em, text centered,
                  rounded corners, minimum height=2em},
    line/.style={draw, -latex}
}%

\title{Uniting Iteration Limits for Mixed-Integer Quadratic MPC}

\author{Luke Fina$^{*}$, Christopher Petersen$^{*}$\thanks{$^{*}$Department of Mechanical and Aerospace Engineering, University
of Florida, Gainesville, FL USA. Emails: \textit{l.fina@ufl.edu,~c.petersen1@ufl.edu. }}}

\begin{document}

\maketitle

\begin{abstract}
Iteration limited model predictive control (MPC) can stabilize a feedback control system under sufficient conditions; this work explores combining a low iteration limit MPC with a high iteration limit MPC for mixed-integer quadratic programs (MIQPs) where the suboptimality is due to solver iteration limits. To combine the two MPCs a hybrid systems controller is developed that ``unites'' two MIQP-MPC solvers where the iteration limits of interest are the branch-and-bound and quadratic programming iteration limits. Asymptotic stability and robustness of the hybrid feedback control system are theoretically derived. Then an interpretable branch-and-bound algorithm and implementable uniting controller algorithm are developed. Finally, the developed algorithms and varying iteration limits are empirically evaluated in simulation for the switching thruster and  minimum thrust spacecraft rendezvous problems.
    
\end{abstract}
\section{Introduction}

Mixed-integer quadratic model predictive control (MIQP-MPC) can solve complex autonomy problems that blend decision making with control \cite{ioan2021mixed}, but are computationally expensive \cite{pia2017mixed}. We reduce compute usage in this paper via switching between two suboptimal MIQP-MPC solvers. There is a growing body of literature where suboptimal solutions are sufficient for a stable MPC feedback control laws \cite{scokaert2002suboptimal,hauswirth2024optimization}. Suboptimal MPC solutions are advantageous in compute-limited autonomous systems, e.g, nonlinear spacecraft rendezvous \cite{behrendt2025time} and spacecraft spin stabilization \cite{ambrosino2025regret}. We model the MPC feedback control loop via a hybrid system model with a `united' controller where the united controllers are the same MPC problem with high and low iteration limits \cite{sanfelice2021hybrid}, where the high limit is globally stabilizing and the low limit is locally stabilizing. This allows for an interpretable branch-and-bound algorithm and implementable uniting control algorithm to solve the MIQP-MPC problem. Specifically, we model switching between branch-and-bound iteration limits and quadratic programming (QP) iteration limits from a hybrid model with conditions for asymptotic stability of the feedback control system. As a result, this paper yields a framework for suboptimal MIQP-MPC via switching the solver iteration limits to user-desired limits.

To the authors' knowledge, this work is the first to model switching between two iteration limited MPC controllers as a united hybrid model. We draw theoretical insights from a recent work highlighting that parametric MIQPs can be viewed as a well-posed hybrid dynamical systems and that there exist iteration limited cases for MIQP-MPC that are theoretically stable \cite{fina2025hybrid}. Furthermore, \cite{ludden2025computational} proposes that there exist controllable computational dynamics for MPC on hardware. This leads to the question: are iteration limits controllable computational dynamics? To explore if computational dynamics are controllable, this paper uses hybrid systems theory \cite{sanfelice2021hybrid}. Hybrid systems are systems that have continuous time and discrete time parts. 

There are few works that align with the goals of this paper. However, inspiration for this paper blends ideas from switching via hybrid systems theory and reference governors. Both \cite{hustig2024uniting,wintz2022global} are hybrid methods that inspire this paper, but neither work focuses on MPC or MIQPs. This paper is similar in its goal to reduce optimization compute with reference governors \cite{kolmanovsky2014reference}, specifically feasibility reference governors enlarge the feasible region of MPC via modifying the reference signal \cite{skibik2021feasibility}.
This paper and reference governors build off existing stable controllers to retain goals encoded in optimization constraints and closed loop stability, but reference governors modify the reference signal while this paper modifies the solver iteration limit of the MPC optimization problem to a user-desired limit. 

% From a hybrid systems perspective, \cite{hustig2024uniting} is similar to this work where they explore switching between different optimization algorithms \cite{hustig2024uniting}. Another similar hybrid systems work switches between certified and uncertified controllers \cite{wintz2022global}. 

The idea of user-desired solver limits comes from \cite{scokaert2002suboptimal} and \cite{behrendt2025time}
that explores the computational viability and dynamical stability of a solver iteration limited case, termed ``time-constrained MPC'', for nonlinear MPC applied to the autonomous rendezvous and proximity problem. Additionally, \cite{morales2025lightweight} explores compute-limited nonlinear MPC computation limits for robotic blimps. In contrast, we focus on MIQPs and switch between a high and low iteration limit with stability results from hybrid systems theory.

For MIQP algorithms, there is a need to distinguish between the effect of integer variable constraint satisfaction and continuous variable constraint satisfaction. Therefore, inspired by the corpus of literature focused on interpretable neural network models, we develop an interpretable branch-and-bound algorithm \cite{belotti2013mixed}.  The interpretable branch-and-bound algorithm allows for clear analysis of integer and continuous variable constraint satisfaction. We view branch-and-bound iteration limits as controlling integer constraint satisfaction and quadratic programming iteration limits as controlling continuous constraint satisfaction.
There are multiple works that develop interpretable MIQP solvers aimed at computationally-limited systems \cite{quirynen2023tailored,stellato2018embedded,arnstrom2023bnb}, but they focus on computational efficiency and speed, e.g., lower memory requirements and reusing previous solutions within the branching tree. Compared to this paper, \cite{quirynen2023tailored,stellato2018embedded,arnstrom2023bnb} are not explicitly derived by control theory or address MIQP solver iteration limits.

In contrast to the works in the literature, we develop a hybrid systems model for switching between high and low iteration limits when solving the MIQP-MPC. Then an implementable control law algorithm is developed (Algorithm~\ref{alg:uniting_control}) along with an interpretable branch-and-bound algorithm (Algorithm~\ref{alg:bnb}) for theoretical stability analysis and interpretable algorithm analysis. The stability analysis is validated in simulation for the following MIQP-MPC problems:
1) switching control spacecraft rendezvous problem \cite{finaScitech2026} 2) minimum thrust spacecraft rendezvous problem \cite{fina2025hybrid}.

The contributions of this work are the following:
\begin{itemize}
  \item A hybrid model and uniting control law for switching branch-and-bound iteration limits and quadratic programming iteration limits with asymptotic stability and robustness guarantees (Section~\ref{sec:hybrid_theory}).
    \item An interpretable branch-and-bound algorithm for solving MIQPs and an implementable hybrid uniting control algorithm (Section~\ref{sec:algs}).
    \item The stability and robustness theoretical guarantees are evaluated in simulation for various branch-and-bound and quadratic programming iteration limited cases applied to the switching thruster spacecraft rendezvous problem (Example~\ref{prob:switching_thrusters}) and minimum thrust spacecraft rendezvous problem (Example~\ref{prob:min_thrust_MIQP}).

\end{itemize}
The main improvements from the conference version \cite{fina2026ACC} are hybrid system theory, an extensive simulation study, and comparison of quadratic programming iteration limits versus branch-and-bound iteration limits.

The rest of this paper is organized as follows. Section~\ref{sec:prelims} presents preliminaries and the MIQP-MPC problem formulations. Section~\ref{sec:hybrid_theory} presents hybrid systems theory preliminaries and stability analysis. Section~\ref{sec:algs} presents an interpretable branch-and-bound algorithm with preliminaries and an implementable uniting control law. Section~\ref{sec:results} presents simulation results for the proposed hybrid uniting control law. Section~\ref{sec:conclusion} concludes. 

\textbf{Notation:} The notation $\mathbb{R}$ denotes the real numbers, $\mathbb{R}_{+}$ denotes the positive real numbers, $\mathbb{Z}$ denotes the integers, $\mathbb{Z}_{+}$ denotes the positive integers, $\mathbb{B}$ denotes the binary variables $\{ 0,1\},$ and $\mathbb{Q}$ denotes the rational numbers. For a vector $\nu \in  \mathbb{R}^{n},$ we write $\nu^{T}$ for its transpose. The notation  $\| \cdot \|_{\infty}$ denotes the $\ell^{\infty}$ norm, $\| \cdot \|_{1}$ denotes the $\ell^{1}$ norm, $\| \cdot \|_{2}$ denotes the $\ell^{2}$ norm, $\mathbb{N}$ denotes an $\ell^{2}$ norm ball around the origin, $\kappa(\cdot)$ denotes a mapping from the optimization vector, $y$, to extract the control part $v$, $M:\mathbb{R}^{n}\rightrightarrows\mathbb{R}^{m}$ denotes a set value mapping \cite{sanfelice2021hybrid}, $I_{m}$ denotes an $m\times m$ identity matrix, and $\setminus$ denotes the set subtraction.

\textbf{Optimization Terminology:} Throughout this work there is a distinction between an arbitrary optimization variable on a computer $y$, a physical state of the system $x,$ and the control of the physical state $u.$ In Section~\ref{sec:results}, $\zeta$ and $v$ are optimization variables on the computer for the physical state $x$ and control of the physical state $u$ respectively. The node $i^{b}$ of a branch-and-bound tree is denoted as $\mathcal{B}(i^{b})$, the branch-and-bound iteration limit at sample time $k$ is denoted $i^{b}_{k},$  $i^{b}_{0}$ denotes an initial feasible branch-and-bound iteration limit, the quadratic programming iteration limit at sample time $k$ is denoted $i^{qp}_{k},$  and $i^{qp}_{0}$ denotes an initial feasible quadratic programming iteration limit.  

\section{Preliminaries and Problem Statement}\label{sec:prelims}

This section presents (1) the main problem of interest to solve, (2) a general form for MIQP model predictive control (MPC) via a parametric MIQP formulation, and (3) a linear time-invariant dynamical system.

At a high level, this paper solves a parametric MIQP-MPC problem (Problem~\ref{prob:MIQP}) at each sampling time $k$ via an interpretable branch-and-bound algorithm (Algorithm~\ref{alg:bnb}) with user-desired branch-and-bound iterations limits ($i^{b})$ and quadratic programming iteration limit ($i^{qp}).$
Concisely, this paper solves the following problem

\begin{problem}\label{prob:low_compute} Solve a parametric MIQP-MPC (Problem~\ref{prob:MIQP}) via an interpretable branch-and-bound algorithm (Algorithm~\ref{alg:bnb}) with $i^{b}_{0}$, $i^{qp}_{0}$, to a user-desired compute limit via iteration limits and maintain closed loop asymptotic stability.
\end{problem}
From \cite{scokaert2002suboptimal}, a warm started feasible solution for MPC implies stability. Therefore, this paper defines a high and low iteration limit where the former finds a feasible solution and the latter is warm started by the high iteration limit solution. The goal of the low iteration limit is to provide a methodology to vary compute via solver limits for a user-desired tolerance.
In this paper, to solve Problem~\ref{prob:low_compute} a hybrid model is developed to switch between a high and low iteration limit for $i^{b}_{k}$ and $i^{qp}_{k}$ at each sample time. In hybrid systems terms, the low and high iteration limits are treated as a locally stable and globally asymptotic controller, respectively, defined further in Section~\ref{sec:hybrid_theory}.

Consider a continuous linear dynamical system of the form
\begin{equation}\label{eq:lin_dyn}
\dot{x}(t)= Ax(t)+Bu(t),
\end{equation}
where $x\in \mathbb{R}^{p}$ denotes the physical state, a sampling of $x(t)$ is defined as $x_{k}=x(k\Delta T)$ for the k-th fixed sampling time, $A\in \mathbb{R}^{q\times p}$, $B\in \mathbb{R}^{q\times r}$ denote state and control matrices, respectively, that take an explicit form in Section~\ref{sec:results},  $~u\in \mathbb{R}^{r}$ denotes the control of the physical system, and in this work $u$ is a feedback law obtained from solving an iteration limited version of Problem~\ref{prob:MIQP} with further details in Section~\ref{sec:hybrid_theory}.

\subsection{Parametric MIQP-MPC Formulation}
The MIQP-MPC of interest takes a compact parametric MIQP form to control Eq.~(\ref{eq:lin_dyn}) at sample time $k.$
\begin{problem}\label{prob:MIQP}
 \begin{align*}
    \psi(y^{*},x_{k})= ~&\underset{y}{\min}  \enskip y^{T}Qy+c^{T}y\\
    &\text{subject to} 
    \\
    &\qquad  y_{0}= x_{k},\\
    &\qquad  Cy \leq b,
    \end{align*}
where  $ \psi(y^{*},x_{k})$ denotes the objective function at the optimal solution $y^{*}$, $y \in \mathbb{R}^{n}$, $y_{1},...,y_{s}\in \mathbb{B}^{s}$ are binary variables that are part of the full vector $y$, $x_{k}$ is defined in Eq.~(\ref{eq:lin_dyn}) sampled at time $k$, $Q\in \mathbb{R}^{n\times n},~c\in \mathbb{R}^{n}$, $y_{0}\in \mathbb{R}^{p}$ is an optimization variable that updates every MPC loop with a new value from $x_{k}$, $C\in \mathbb{Q}^{m\times n},~b\in \mathbb{R}^{m}.$  
\end{problem}
\noindent Problem~\ref{prob:MIQP} is formulated in a parametric MIQP form, common for MPC and control barrier function MIQP formulations, e.g, online task allocation  \cite{notomista2021resilient}, signal temporal logic \cite{yang2020continuous}, and mixed-logical dynamical systems \cite{bemporad1999control}.

% The optimal mapping for Problem~\ref{prob:MIQP} can be reformulated to the following set value mapping
% \begin{equation}
%     \phi(x_{k}): \mathbb{R}^{p}\rightrightarrows\mathbb{R}^{n},
% \end{equation}
% where $\phi(x_{k})$ will be of importance in the hybrid systems analysis below.  

The feasible constraint Problem~\ref{prob:MIQP} can be reformulated to the following set value mapping
\begin{equation}\label{eq:const_set_value_map}
    h(x_{k},b): \mathbb{R}^{p+m}\rightrightarrows\mathbb{R}^{n},
\end{equation}
where $h(x_{k},b)$ is defined further for specific iteration limits in Section~\ref{sec:setvalue_iter_limit}. The feasible constraint mapping is defined over the feasible set next.

\begin{proposition}\label{prop:hybrid_conds_sat}
    The feasible constraint set value mapping, $h(x_{k},b),$ is continuous on the feasible parameter set, where the feasible parameter set is defined as
\[
\Gamma := \{ x_{k}\in \mathbb{R}^{p},~b\in \mathbb{R}^{m}~ \vert~ h(x_{k},b) \neq \emptyset\}.
\]
\end{proposition}
In brief, for rational constraint and objective function matrices with additional mild technical conditions, the constraint set mapping is well behaved \cite{fina2025hybrid}[Lemma 4].

\subsection{Iteration Limited Constraint Mappings}\label{sec:setvalue_iter_limit}
We model the feasible constraint mapping in Eq.~(\ref{eq:const_set_value_map}) for both iteration limited cases via a local and global controller framework, denoted $h_{0}$ and $h_{1}$. The two iteration limits of interest are the branch-and-bound iteration limit and quadratic programming iteration limit. It is outside the scope of this work, but the hybrid model framework in Section~\ref{sec:hybrid_theory} is general enough for nontraditional solver limits, e.g., memory limit \cite{fina2025hybrid}. 

Since the constraint vector $b$ and iteration limits are constant, the feasible constraint mapping for the branch-and-bound iteration limited case is defined in concise notation as
\begin{align}
    h^{b}_{0}(x_{k}):=h(x_{k},b,\underline{i}^{b},i^{qp}_{0}),\\
    h^{b}_{1}(x_{k}):=h(x_{k},b,\bar{i}^{b},i^{qp}_{0}),
\end{align}
where $\underline{i}^{b},~\bar{i}^{b}$ denote the low and high branch-and-bound iteration limit, respectively. The feasible constraint mapping for the quadratic programming iteration limited case is defined as
\begin{align}
    h^{qp}_{0}(x_{k}):=h(x_{k},b,i^{b}_{0},\underline{i}^{qp}),\\
    h^{qp}_{1}(x_{k}):=h(x_{k},b,i^{b}_{0},\bar{i}^{qp}),
\end{align}
where $\underline{i}^{qp},~\bar{i}^{qp}$ denote the low and high quadratic programming iteration limit, respectively. 

For notational simplicity in the rest of this work, the local controllers $h_0^b (x_k),~h_0^{qp}(x_k)$ and global controllers $h_1^b (x_k),~h_1^{qp}(x_k)$ are denoted $h_{0}(x_{k})$ and $h_{1}(x_{k})$ for both iteration limited cases, as the switching logic analysis via hybrid system theory is the same for either iteration limit. Even though the hybrid analysis is the same, results for the different behaviors they induce are reflected in Section~\ref{sec:results}.

\section{Hybrid Systems Model}\label{sec:hybrid_theory}
In this section, Problem~\ref{prob:low_compute} is formulated into a hybrid system model with the iteration limited controllers as a ``united'' controller \cite{sanfelice2021hybrid}[Chapter 4]. We denote a hybrid state $\chi(t,j)$, where $t$ are flow times (i.e., time arcs where the system flows continuously) and $j$ are jump times (i.e., discontinuities where the hybrid system jumps states). A general hybrid system is defined as $\mathcal{H}=(C,F,D,G)$ with

\begin{align*}
&\chi \in C, \quad \dot{\chi} \in  F(\chi),\\
&\chi \in D, \quad \chi^{+} \in G(\chi),
\end{align*}
where $\chi\in \mathbb{R}^{s}$ denotes the hybrid system state, $F$ denotes the set value flow map, $C$ denotes the flow set, $G$ denotes the set value jump map, $D$ denotes the jump set \cite{goebel2009hybrid}[Theorem 6.8], and $\dot{\chi},~\chi^{+}$ denote the evolution of the hybrid state in the flow and jump map, respectively. We model controlling the dynamical system with iteration limited MPC (Problem~\ref{prob:MIQP}) in continuous time and switching between the iteration limited MPC controllers in discrete time. In hybrid system terminology, this means the controlled dynamics, $\dot{x}=Ax+Bu$, flow and the switching law jumps. 

Regularity, i.e., conditions for a well behaved hybrid system,  can be defined by the hybrid basic conditions,
\begin{definition} (Hybrid Basic Conditions)\label{def:hybrid_conds}
\cite{goebel2009hybrid}
\begin{enumerate}
    \item $C$ and $D$ are closed subsets of $\mathbb{R}^{s}$;
    \item The flow map $F$ is outer semi-continuous and locally bounded relative to $C$, $C\subset \text{dom} F$, and $F(\chi)$ is convex for every $\chi\in C$;
    \item The jump map $G$ is outer semi-continuous and locally bounded relative to $D$, $D\subset \text{dom} G$.
\end{enumerate}
\end{definition}

The hybrid system, $\mathcal{H}_{s},$ presented next unites the local and global controllers defined in Section~\ref{sec:setvalue_iter_limit}. The linear dynamics defined in Eq.~(\ref{eq:lin_dyn}) are controlled via the united controller. The hybrid system model $\mathcal{H}_{s}$ models sufficient criteria for stable switching without exhaustively modeling all possible control features like sampling time and model mismatch because hybrid systems under sufficient conditions can treat these control features as perturbations to $\mathcal{H}_{s}$.

The flow map and flow set for $\mathcal{H}_{s}$ are defined as
\begin{align}\label{eq:flow_dynamics}
F(\chi):=\begin{bmatrix} \dot{x} = Ax+B h_{q}\big(x\big) \\ \dot{q} = 0\end{bmatrix} \qquad \chi \in C:= C_{0} \cup C_{1}.
\end{align}
The jump map and jump set are defined as
\begin{align}
G(\chi):= \begin{bmatrix}x^{+} = x \\
q^{+} = 1- q\end{bmatrix} \qquad \chi \in D:= D_{0} \cup D_{1}.
\end{align}

The sets $C_{0},~C_{1},~D_{0},$ and $D_{1}$ are defined as
\begin{align}
    C_{0}:= \mathcal{U}_{0} \times \{ 0\},~ C_{1}:= \mathcal{T}_{1} \times \{ 1\}, \\
    D_{0}:= \mathcal{T}_{0} \times \{ 0\},~ D_{1}:= \mathcal{T}_{1} \times \{ 1\},
   \end{align}
where the hybrid state is defined as $\chi:=(x,q)\in \mathbb{R}^{p}\times Q$, $Q\in \{ 0,1\},$ and $\mathcal{U}_{0},~\mathcal{T}_{0},~\mathcal{T}_{1}$ are defined further in Section~\ref{sec:supervisor_control}.

The stabilizing set for $\mathcal{H}_{s}$ is defined as
\[
\mathcal{A}:= X \times  \{ 0 \},
\]
where $X$ is a compact set of equilibrium points when controlled by solving Problem~\ref{prob:MIQP} and $\{ 0 \}$ denotes the switching controller equilibrium, i.e., switching to the local controller $h_{0}$. The switch between $h_{0}$ and $h_{1}$ is controlled via a supervisory algorithm defined by Lyapunov functions that measure deviation from the set $\mathcal{A}$ with further details in Section~\ref{sec:supervisor_control}. The control logic of $\mathcal{H}_{s}$ is illustrated in Figure~\ref{fig:feedback_law_concept}.

\begin{figure}[!ht]

\centering
\begin{tikzpicture}[auto, node distance=2cm,>=latex']
    \node [block] (feasIter) {MPC Update with High Iteration Limit};
    \node [block, above of=feasIter] (plant) {$\dot{x}=Ax+Bu$};
     \node [block, below of=feasIter] (infeasIter) {MPC Update with Low Iteration Limit};

  \path [line] (feasIter) --++ (-3.25cm,0cm) |- (plant);
  \path [line] (infeasIter) --++ (-3.25cm,0cm) |- (plant);

   \path [line] (plant)--++ (3cm,0cm)  |- (feasIter);
   \path [line_label] (feasIter)-- node{$q=1$} (3cm,0cm) ;
\path [line] (plant)  --++ (3cm,0cm)|- (infeasIter) ;
\path [line_label] (infeasIter)-- node{$q=0$} (3cm,-2cm);

\path[line_label](feasIter) -- node{$u=h_{1}(x)$} (-3.25cm,0cm);

\path[line_label](infeasIter) -- node{$u=h_{0}(x)$} (-3.25cm,-2cm);

\end{tikzpicture}
\caption{Feedback Loop for Hybrid Uniting Control Law}
\label{fig:feedback_law_concept}
\end{figure}
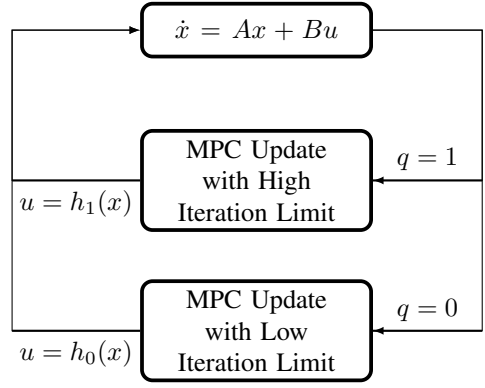

\subsection{United Control Law}\label{sec:supervisor_control}

Two Lyapunov-based measures of $\mathcal{H}_{s}$ converging to the stabilizing set $\mathcal{A}$ are objective function and constraint based where the first optimization-based term occurs at specific sampling times and the second term is physical state based. The optimization updates can be modeled as occurring in discrete jump time at specific sample times as done in \cite{fina2025hybrid}, but it does not affect the analysis of switching iteration limited controllers so it is omitted. An objective function based Lyapunov function can be defined as
\begin{equation}
    V_{obj}(x)=\theta\vert \psi(y^{*},x_{k})-\psi(y^{*},x_{k-1}) \vert + x^{T}\Gamma x,
\end{equation}
 where $y^{*}$ is the optimum of solving Problem~\ref{prob:MIQP} evaluated at sample time $k$ and $k-1$ respectively, $x$ denotes the physical state as defined in Eq.~(\ref{eq:lin_dyn}), $x_{k}$ denotes the sampled $x$,  $\Gamma:=\sigma I_{p}$ and $~\sigma \in \mathbb{R}_{+}, $ denotes a penalty term. Using the objective function as a measure of stability for MPC as done with $V_{obj}$ is a classic method to prove stability of a MPC control law \cite{mayne2000constrained}. However, a sufficient condition for MPC stability that aligns with the hybrid model in $\mathcal{H}_{s}$ is constraint based \cite{scokaert2002suboptimal}. A feasible constraint based Lyapunov function is defined as
\begin{equation}
    V_{feas}(x) = \theta \vert \vert Cy^{*}-\bar{b} \vert \vert_{\infty} + x^{T}\Gamma x,
\end{equation}
where $\theta \in \mathbb{R}_{+}$ denotes a penalty term, $y^{*}$ denotes the solution to Problem~\ref{prob:MIQP} evaluated at $x_{k}$, $\bar{b}:=[x_{k}^{T}~~b^{T}]^{T},$ and the $\ell^{\infty}$ norm is a practical termination criteria to measure feasibility in optimization solvers \cite{stellato2020osqp}.

Next, $V_{obj}$ and $V_{feas}$ are used to construct the switching law control sets for $\mathcal{H}_{s},$ specifically $\mathcal{T}_{0},~\mathcal{T}_{1}, $ and $\mathcal{U}_{0}$ are defined as
\begin{align}
    \mathcal{T}_{0}:=\{  V_{p}(x)  \leq c_{p,0}\}, \label{eq:T0_lyap}\\
     \mathcal{U}_{0}:= \{   V_{p}(x)  \leq c_{p,1}\},\label{eq:U0_lyap}\\
    \mathcal{T}_{1}:= \{V_{p}(x)  \geq c_{p,1}\} \label{eq:T1_lyap},
\end{align}
where $~c_{p,0}<c_{p,1}$, $p\in \{ obj,feas\}$ and the set $\mathcal{U}_{0}$ reduces Zeno behavior when switching between $h_{0}$ and $h_{1}$.

\subsection{Stability Analysis}\label{sec:stability_analysis}

Asymptotic stability of $\mathcal{H}_{s}$ to $\mathcal{A}$ is presented under the next assumption.
\begin{assumption}\label{assum:asym_control}\cite{sanfelice2021hybrid}[Assumption 4.3] Given $x^{*}\in \mathbb{R}^{p}$ and the continuous dynamics, there exists a closed set $\mathcal{U}_{0}\subset \mathbb{R}^{p},$
which contains an open neighborhood of $x^{*},$ and a closed set $\varepsilon_{0}\subset \mathbb{R}^{p}$ such that 
\begin{enumerate}
    \item A state-feedback law $h_{1}: \mathbb{R}^{p} \to \mathbb{R}^{r}$ such that the dynamics controlled by $h_{1}$ is such that $x^{*}$ is asymptotically stable with basin of attraction containing $\mathcal{U}_{0},$
    \item A state-feedback law $h_{0}: \mathbb{R}^{p} \to \mathbb{R}^{r}$ guaranteeing that every solution $x$ to the linear dynamics controlled by $h_{0}$ starting from $\mathcal{U}_{0} \setminus \varepsilon_{0}$ reaches $\varepsilon_{0}$ after a finite amount of flow time or converges to it as $t$ tends to $\infty,$
    \item There exist positive constants $\delta_{0},~\delta^{c}_{0}$ and a closed set $\mathcal{T}_{0}$ satisfying
    \[
    \varepsilon_{0}+\delta^{c}_{0}\mathbb{N} \subset \mathcal{T}_{0} \qquad \mathcal{T}_{0}+2\delta_{0}\mathbb{N}\subset \mathcal{U}_{0}.
    \]
\end{enumerate}

\end{assumption}
The set $\mathcal{T}_{1}$ in Eq.~(\ref{eq:T1_lyap}) is equivalently defined as $\mathcal{T}_{1}=\mathbb{R}^{p} \setminus \mathcal{T}_{0}$ and by definition could lead to Zeno behavior on the boundary between the Lyapunov function sets $\mathcal{T}_{1}$ and $\mathcal{T}_{0}$ without a buffer set. Therefore, the set $\mathcal{U}_{0}$ is a user-desired buffer set between $\mathcal{T}_{0}$ and $\mathcal{T}_{1}$ to prevent Zeno behavior. At a high level, $\varepsilon_{0}$ is a set around the solution $x^{*}$ to allow for minor perturbations from the exact $x^{*}$ value. 
Figure~\ref{fig:concept_switching_law_sets} illustrates the relationship between, the sets, $\mathcal{T}_{0},$ $\mathcal{U}_{0},$ and $\varepsilon_{0}$.

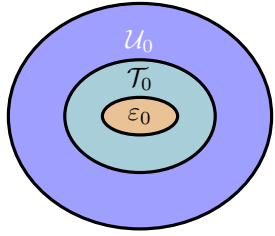
\begin{figure}[!ht]
\centering
\begin{tikzpicture}

  % \fill[fill=orange!90, fill opacity = 0.5] (0,0) ellipse (2.5 and 2.0);
  \filldraw[color=black, fill=blue!50, fill opacity = 0.75, very thick] (0,0) ellipse (1.75 and 1.5);
  \filldraw[color=black,fill=green!30, fill opacity = 0.5, very thick] (0,0) ellipse (1.0 and 0.75);
  \filldraw[color=black, fill=orange!50, fill opacity = 0.75,very thick] (0,0) ellipse (0.5 and 0.25);
  
  % Add labels
  \node[black] at (0,0) {$\varepsilon_{0}$};
  \node[black] at (0,0.5) {$\mathcal{T}_{0}$};
  \node[white] at (0,1) {$\mathcal{U}_{0}$};
  % \node[black] at (0,1.75) {$\mathcal{T}_{1}$};
  
\end{tikzpicture}
\caption{Supervisory Law Sets}
\label{fig:concept_switching_law_sets}
\end{figure}

Assumption~\ref{assum:asym_control}.1 is reasonable for MPC when the terminal set is well formulated \cite{mayne2000constrained}. The MPC controller with a high iteration limit is $h_{1}.$  Assumption~\ref{assum:asym_control}.2 is reasonable when the MPC controller is warm started with a feasible initial solution, which is true by control law construction when $h_{1}$ warm starts $h_{0}$ \cite{scokaert2002suboptimal}.
Assumption~\ref{assum:asym_control}.3 is a technical condition to ensure the switching sets are well defined and future work will define weaker theoretical conditions.

Before proving asymptotic stability of $\mathcal{H}_{s}$ to $\mathcal{A},$ a sufficient condition for solutions to $\mathcal{H}_{s}$ to be well behaved is when $\mathcal{H}_{s}$ satisfies Definition~\ref{def:hybrid_conds}.
\begin{lemma}\label{lem:basic_conditions} Let Assumption~\ref{assum:asym_control}, hold. Then $\mathcal{H}_{s}$ satisfies the hybrid basic conditions for the constraint mapping defined in Proposition~\ref{prop:hybrid_conds_sat}.
\end{lemma}
\begin{IEEEproof}
The set $\mathcal{T}_{1}$ in Eq.~(\ref{eq:T1_lyap}) is equivalently defined in terms of $\mathcal{T}_{0}$ as $\mathcal{T}_{1}= \mathbb{R}^{p} \setminus \mathcal{T}_{0}.$ By construction and closedness of $\mathcal{T}_{0}$ from Assumption~\ref{assum:asym_control}.3,  the sets $C_{0},~C_{1},~D_{0},$ and $D_{1}$ are closed. Then, $C$ and $D$ are closed because they are defined as the finite union of closed sets. By construction, the maps $G$ and $F$ are continuous. Finally, the output map $h$ is continuous since $h_{0}$ and $h_{1}$ are continuous constraint set value mappings of $(x,q)$ as defined in Proposition~\ref{prop:hybrid_conds_sat}. 
\end{IEEEproof}

Next, Lemma~\ref{lem:basic_conditions} is used to prove asymptotic stability of $\mathcal{H}_{s}$ to the set $\mathcal{A}.$

\begin{proposition}\label{prop:stable_hybrid_system}
    Let Assumption~\ref{assum:asym_control} hold, then the set $\mathcal{A}$ is globally asymptotically stable and robust in the hybrid sense additionally every maximal solution to $\mathcal{H}_{s}$ from $C \cup D$ is complete.
\end{proposition}
\begin{IEEEproof}
    For brevity, the proof follows directly \cite{sanfelice2021hybrid}[Theorem 4.6] with the change that the hybrid basic conditions are satisfied by Lemma~\ref{lem:basic_conditions}.
\end{IEEEproof}

Now that asymptotic stability and robustness have been theoretically established, the rest of this work constructs algorithms to solve Problem~\ref{prob:MIQP} and implement the supervisory control law in $\mathcal{H}_{s}.$

\section{Algorithms}\label{sec:algs}
  In this section, an interpretable branch-and-bound solver with an explicit pruning algorithm and then an implementable uniting control algorithm is presented. The notation $\tilde{y}$ is introduced to denote the best current solution in Algorithm~\ref{alg:bnb} and Algorithm~\ref{alg:prune} evaluted at $x_{k}$.

\subsection{Branch-and-Bound Preliminaries}
Problem~\ref{prob:MIQP} is solved via a branch-and-bound algorithm (Algorithm~\ref{alg:bnb}). Formally, branch-and-bound algorithms for Problem~\ref{prob:MIQP} solve parametric continuous quadratic programs of the following form
\begin{problem}\label{prob:QP}
 \begin{align*}
    \phi(y^{*},x_{k},\mathcal{B}(i^{b}_{k}))= ~&\underset{y}{\min}  \enskip y^{T}Qy+c^{T}y\\
    &\text{subject to} 
    \\
    &\qquad  y_{0}= x_{k},\\
    &\qquad  Cy \leq b,\\
    &\qquad Wy \leq w(\mathcal{B}(i^{b}_{k})),
    \end{align*}
where $\phi(y^{*},x_{k},\mathcal{B}(i^{b}_{k}))$ denotes the objective function at the optimal solution $y^{*}$, $y$ is defined in Problem~\ref{prob:MIQP}, $y_{1},...,y_{s}\in \{y_{j}\in \mathbb{R}^{s}~ \vert~ 0\leq y_{j} \leq 1 \}$ are relaxed binary variables that are part of the full vector $y$, $\mathcal{B}(i^{b}_{k})$ denotes a node in a branching tree at branch-and-bound iteration $i^{b}_{k}$, $x_{k}$ is defined in Eq.~(\ref{eq:lin_dyn}), $y_{0}$, $Q,$  $C,~b$ are defined in Problem~\ref{prob:MIQP}, and $W\in \mathbb{R}^{\ell\times n}$ $~w\in \mathbb{R}^{\ell}$ are constraints added during the branch-and-bound iterations with more details on constraint construction in Section~\ref{sec:algs}.
\end{problem}

Problem~\ref{prob:QP} is solved in Algorithm~\ref{alg:bnb} with FBstab \cite{liao2020fbstab} where no assumption on constraint qualification is made. FBstab is open source, combines the proximal point algorithm with a Newton-type method, and is quite general because it can handle degenerate quadratic programs and positive semi-definite objective function matrices \cite{liao2020fbstab}. This is advantageous for the proposed branch-and-bound solver to quickly find infeasible QPs and handle well-posed MIQPs as defined in \cite{bemporad1999control}.

\begin{assumption}
    The Hessian matrix $Q$ is symmetric and positive semidefinite.
\end{assumption}
The above assumption is sufficient to guarantee a solution exists, which is used within the FBstab computation of QP iteration updates.
\begin{remark}
    FBstab is advantageous for mixed integer quadratic MPC, because binary terms do not need to be explicitly penalized in the matrix $Q.$ Meaning FBstab can solve QP subproblems where the MIQPs are well posed as defined in \cite{bemporad1999control}.
\end{remark}

Algorithm~\ref{alg:bnb} below solves the parametric QP (Problem~\ref{prob:QP}) and then prunes the branch-and-bound tree following the logic in Algorithm~\ref{alg:prune}. Otherwise, the branch-and-bound tree splits on each binary variables defined in Problem~\ref{prob:MIQP} until there are no more binary variables or the branch-and-bound iteration limit is reached. The lines 8 and 11 of Algorithm~\ref{alg:bnb} update a stored vector in memory of the best $y.$ Finally, lines 9-12 of Algorithm~\ref{alg:bnb} handle the case where the branch-and-bound iteration limit is reached before solving the last nodes. 

To note for practitioners, splitting for binary variables in this work means creating two nodes $\mathcal{B}(i^{b}_{k}+1)$ and $\mathcal{B}(i^{b}_{k}+2)$ where $\mathcal{B}(i^{b}_{k}+1)$ has an additional constraint $0\leq y_{j}\leq 0$ and $\mathcal{B}(i^{b}_{k}+2)$ has an additional constraint $1\leq y_{j}\leq 1,$ where $j$ denotes the current binary variable being branched. This is in contrast to equality constraints because inequality constraints are able to be satisfied in fewer FBstab iterations than equality constraints. After a split, Problem~\ref{prob:QP} is solved for $\mathcal{B}(i^{b}_{k}+1)$ and $\mathcal{B}(i^{b}_{k}+2)$ following the logic in Algorithm~\ref{alg:bnb} until reaching max branch-and-bound iteration at sample time $k$ denoted $\bar{i}^{b}_{k}.$

Classically, Algorithm~\ref{alg:bnb} is termed a ``T-cut suboptimal branch-and-bound method", where the branch-and-bound ends after a finite node limit is reached \cite{ibaraki1983using}. We provide a similar form, but branches the last nodes unlike the classic form. It is worth noting Algorithm~\ref{alg:bnb} can be improved computationally in numerous ways, for example warm starting nodes \cite{marcucci2020warm} and early QP termination \cite{liang2020early}, however this is outside the scope of this work. A branch-and-bound tree formed by Algorithm~\ref{alg:bnb} is illustrated in Figure~\ref{fig:bnb-example} where $\phi(\mathcal{B}(\cdot))$ denotes solving Problem~\ref{prob:QP} at branch-and-bound iteration $\cdot$ for notational brevity.

\begin{figure}
\centering
\includegraphics[scale=0.75]{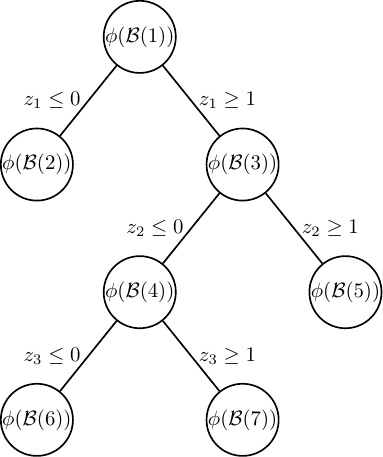}
\caption{Example of a branch-and-bound tree via Algorithm~\ref{alg:bnb}.}
\label{fig:bnb-example}
\end{figure}

\begin{algorithm}\caption{Iteration Limited Branch-and-Bound}\label{alg:bnb}
\KwIn{$i^{b}_{k}$ and $i^{qp}_{k}$ from Problem~\ref{prob:low_compute}, $Q,~C,~b$ from Problem~\ref{prob:MIQP}, $A,~B$, from Eq.~(\ref{eq:lin_dyn}), and $x_{0}$ from Eq.~(\ref{eq:lin_dyn}), $\tilde{y},~\tilde{\psi}$ are initialized to sufficiently large values}
\KwOut{$\tilde{y}$}
\For{$i^{b}_{k}=0,1,2,...,\bar{i}^{b}_{k}$}{
Solve Problem~\ref{prob:QP} for $\mathcal{B}(i^{b}_{k})$ with $i^{qp}_{k}.$

Prune $\mathcal{B}$ via Algorithm~\ref{alg:prune}

\If{$\mathcal{B}(i^{b}_{k})\neq \emptyset$}{Branch on next binary variable}
}
\If{$\mathcal{B}(\bar{i}^{b}_{k}-1)\neq \emptyset$}
{Solve Problem~\ref{prob:QP} for $\mathcal{B}(\bar{i}^{b}_{k}-1)$ with $i^{qp}_{k}.$  \\ Update $\tilde{y}.$}
\If{$\mathcal{B}(\bar{i}^{b}_{k})\neq \emptyset$}
{Solve Problem~\ref{prob:QP} for $\mathcal{B}(\bar{i}^{b}_{k})$ with $i^{qp}_{k}.$ \\ Update $\tilde{y}.$}
\end{algorithm}

Algorithm~\ref{alg:prune} prunes the branch-and-bound tree and updates the best stored solution,~$\tilde{y}$. The depth-first and best-first pruning logic are pruning heuristics that affect the branching prioritization of either the first feasible solution with depth-first or the best solution measured via the objective function with best-first. It is outside the scope of this work to fully analyze these two methods, but in simulation Example~\ref{prob:switching_thrusters} uses the best-first and Example~\ref{prob:min_thrust_MIQP} uses depth-first pruning logic.

\begin{algorithm}\caption{Prune Logic}\label{alg:prune}
\KwIn{$i^{b}_{k}$ from Problem~\ref{prob:low_compute}, $\mathcal{B},~\tilde{y},~\tilde{\psi}$ from Algorithm~\ref{alg:bnb}}
\KwOut{$\tilde{y},~\mathcal{B}$}

\If{Problem~\ref{prob:QP} is infeasible}{$\mathcal{B}(i^{b}_{k})=\emptyset$}
\Else{Update $\tilde{y}$ \\ \If{Depth-First Pruning}{$\mathcal{B}(i^{b}_{k}+1)=\emptyset$}

\If{Best-First Pruning}{
\If{$\psi(\tilde{y}_{1,...,s},x_{k})>\tilde{\psi}$}{$\mathcal{B}(i^{b}_{k})=\emptyset$}
\Else{$\tilde{\psi}=\psi(\tilde{y}_{1,...,s},x_{k})$}
}}

\end{algorithm}

In practice, the control loop can be implemented via a sample-and-hold formulation as demonstrated in Algorithm~\ref{alg:uniting_control} where $k \in K$ denote a sampling time $k$ in a set of sample times $K\in \mathbb{Z}_{+}$, the terms $i_{k}^{b},~i_{k}^{qp}$ are defined in Problem~\ref{prob:low_compute}, and the terms $h_{0}(\cdot),~h_{1}(\cdot)$ are defined in Section~\ref{sec:setvalue_iter_limit}.

\begin{algorithm}\caption{Sampled Uniting Control}\label{alg:uniting_control}
\KwIn{$i^{qp}_{k}=i^{qp}_{0},~i^{b}_{k}=i^{b}_{0}$, $q=1,$ $x=x_{0}$}
\KwOut{$u=h_{q}(x_{k})$}
\For{$k\in K$}{

\If {Branch-and-Bound Iteration Limited Case}{

\If {$x_{k}\in \mathcal{U}_{0}$ and $q=0$}{$u=\kappa(h_{0}(x_{k},i^{b}_{k}))$\\
}

\If {$x_{k}\in \mathcal{T}_{1}$ and $q=1$}{$u=\kappa(h_{1}(x_{k},i^{b}_{k}))$}

\If{$x_{k} \in \mathcal{T}_{1}$ and $q=0$}{$q=1,~i^{b}_{k}=\bar{i}^{b} $}
\If{$x_{k} \in \mathcal{T}_{0}$ and $q=1$}{$q=0,~i^{b}_{k}=\underline{i}^{b}$}
}

\If {Quadratic Programming Iteration Limited Case}{

\If {$x_{k}\in \mathcal{U}_{0}$ and $q=0$}{$u=\kappa(h_{0}(x_{k},i^{qp}_{k}))$\\
}

\If {$x_{k}\in \mathcal{T}_{1}$ and $q=1$}{$u=\kappa(h_{1}(x_{k},i^{qp}_{k}))$}

\If{$x_{k} \in \mathcal{T}_{1}$ and $q=0$}{$q=1,~i^{qp}_{k}=\bar{i}^{qp} $}
\If{$x_{k} \in \mathcal{T}_{0}$ and $q=1$}{$q=0,~i^{qp}_{k}=\underline{i}^{qp}$}
}

}
\end{algorithm}

\section{Simulation Results}\label{sec:results}
In this section, the hybrid model from Section~\ref{sec:hybrid_theory} and algorithms developed in Section~\ref{sec:algs} are evaluated in simulation for the switching thruster and minimum thrust control spacecraft rendezvous problems \cite{finaScitech2026,fina2025hybrid}. Comparisons are made between fixed iteration limits for Algorithm~\ref{alg:uniting_control} with various iterations limits and FBstab as the QP solver. Finally, practical guidelines are outlined for finding values of $c_{p,0}$ and $c_{p,1},$ which are the switching values introduced in Eqs~(\ref{eq:T0_lyap},\ref{eq:U0_lyap},\ref{eq:T1_lyap}). 

The initial position for both spacecraft rendezvous and proximity problems is $x_{0}=[6800,~0,~0,~0,~-15.368  ,~0]^{T}$, the optimization is modeled in YALMIP \cite{lofberg2004yalmip}, the optimal solution is calculated via Gurobi with a two second time limit \cite{gurobi_MIP}, the big-M term for Example~\ref{prob:switching_thrusters} is $0.1$, the MPC horizon is $N=15$, and the sample time is $5$ minutes for slow spacecraft dynamics with more details in Appendix~\ref{sec:appendix}. The switching law terms are $c_{obj,0}=100,~c_{obj,1}=1000,~c_{feas,0}=200,~c_{feas,1}=300$, the Lyapunov penalty terms are $\theta =1,~\sigma=1^{-5}$ for $V_{obj}$, the Lyapunov penalty terms are $\theta =1^{-3},~\sigma=1^{-5}$ for $V_{feas}$, and $\bar{i}^{b}=20,~\bar{i}^{qp}=100$ when switching branch-and-bound iteration limits and quadratic programming iteration limits, respectively. Both examples denote dynamic and control optimization variables as $\zeta,~v$, respectively, and optimization binary variables as $z$. The optimization variables can be stacked to take the form of Problem~\ref{prob:MIQP} as $y:=[\zeta^{T},~v^{T},z^{T}]^{T}.$

The switching thrusters problem is as follows
\begin{example}[Spacecraft Rendezvous and Proximity with Switching Thrusters]\label{prob:switching_thrusters}
 \begin{align*}
   \underset{\zeta,v_{1},v_{2},z}{\text{minimize}} & \enskip \alpha_{v_{1}}\vert \vert v_{1}\vert \vert_{2}+\alpha_{v_{2}}\vert \vert v_{2}\vert \vert_{2}+\alpha_{1}\vert \vert \zeta \vert \vert _{2}\\
    \text{subject to} & \enskip 
    \zeta_{0} = x_{k},\\&\zeta_{k+1}=A\zeta_{k}+B_{1}v_{k,1}+B_{2}v_{k,2}, \\
     & -M z_{k}\leq  v_{k,1}  \leq M z_{k},\\
     & -M (1-z_{k}) \leq v_{k,2}  \leq M (1-z_{k}),\\
     &  \zeta_{N}  = \textbf{0}, \\
     &\zeta\in Z ,v_{1}\in V,v_{2}\in V, 
    \end{align*}
    where $\zeta\in\mathbb{R}^{pN}$ is the state vector defined as $\zeta:=[\zeta_{1}^{T},...,\zeta_{N}^{T}]^{T}$ over the prediction horizon, $N$ is the MPC prediction horizon, $v_{1}\in\mathbb{R}^{rN},v_{2}\in\mathbb{R}^{rN}$ are the control parts for mode 1 and 2 for the MPC horizon, $z\in \mathbb{B}^{rN}$ is vector of binary variables for the MPC horizon, $\alpha_{v_{1}},\alpha_{v_{2}},\alpha_{1}\in \mathbb{R}_{+}$ are user-desired scalar penalty terms,$\zeta_{0}$ is the initial state for the optimization variables take from $x_{k}$ defined in Eq (\ref{eq:lin_dyn}), $x_{k}$ is defined in Eq (\ref{eq:lin_dyn}) sampling time $k$, ~$\zeta_{k}\in \mathbb{R}^{p}$, $A\in \mathbb{R}^{p\times p}$ is a matrix for linear dynamics,  $B_{1}\in \mathbb{R}^{p\times r}, B_{2}\in \mathbb{R}^{p\times r}$ which are all defined further in Appendix~\ref{sec:appendix},  $M$ is a big-M constant, $v_{k,1}\in \mathbb{R}^{r}$ is the control for subsystem 1 and $v_{k,2}\in \mathbb{R}^{r}$ is the control for subsystem 2 at sampling time $k$,   $z_{k}\in \mathbb{B}^{r}$, $\zeta_{N}$ is the terminal state on an MPC horizon $N$, and $Z,V_{1},V_{2}$ are compact set constraints for $\zeta$ and $v_{1},v_{2}$, respectively.
\end{example}

 The minimum thrust problem is defined next where the $\ell^{1}$ norm control constraint is nonconvex and can be reformulated into a mixed-integer form seen in \cite{fina2025hybrid}, the optimization state penalty matrix is $P=10^{-7}I_{p}$, and optimization control penalty matrix is $R=10^{2}I_{p}$.
\begin{example}[Spacecraft Rendezvous and Proximity with Minimum Thrust Control]\label{prob:min_thrust_MIQP}
    \begin{align*}
    &\underset{\zeta,v,z}{\min}  \enskip \zeta^{T}P\zeta+v^{T}Rv\\
    &\text{subject to} 
    \enskip \zeta_{0} = x_{k},\\  
     &\qquad \zeta_{k+1} = A\zeta_{k}+Bv_{k}, \\
     &\qquad  z_{k}v_{\min}\leq \vert \vert v_{k}\vert \vert_{1} \leq v_{\max}z_{k}, \\
     &\qquad \zeta_{N}\in Z_{\text{terminal}},
    \end{align*}
where $\zeta \in \mathbb{R}^{pN}$ is defined over the prediction horizon, $N$ is the MPC prediction horizon, $v\in \mathbb{R}^{rN}$,$z\in \mathbb{B}^{rN}$, $P\in \mathbb{R}^{pN\times pN}$ is a time-invariant cost on state, $R\in \mathbb{R}^{rN\times rN}$ is a time-invariant cost on state, $\zeta_{0}$ is the initial state for the optimization variables take from $x_{k}$ defined in Eq (\ref{eq:lin_dyn}), $\zeta_{k}\in \mathbb{R}^{p}$, $A\in \mathbb{Q}^{p\times p}$ are the rendezvous dynamics defined in detail in Appendix~\ref{sec:appendix}, $B\in \mathbb{R}^{p\times r}$ is an electric control thrust defined in detail in Appendix~\ref{sec:appendix}, $v_{\min}\in \mathbb{R},~v_{\max}\in \mathbb{R}$ are lower and upper bounds on control $v$, $z_{k} \in \mathbb{B}$, $\zeta_{N}$ is the terminal state on an MPC horizon $N$, and $Z_{\text{terminal}}$ is a compact terminal set. 
\end{example}

\subsection{Algorithm~\ref{alg:uniting_control} Compared to a Fixed Iteration Limit}\label{sec:comparionOptimal}
In this section, the tracking error between the optimal trajectory and suboptimal trajectories via Algorithm~\ref{alg:uniting_control} for fixed iteration limits, $V_{obj},$ and $~V_{feas}$ are measured with the $\ell^{2}$ norm.

\subsubsection{Results for Switching Control}

In Figure~\ref{fig:fixed_bnb_ST}, fixed branch-and-bound iteration limits of $i^{b}_{0}=2,5,10,20$ converge to a mild tracking error after 30 sample times. In Figure~\ref{fig:fixed_qp_ST}, fixed quadratic programming iteration limits of $i^{qp}_{0}=1,4,5,7,10$ have larger tracking errors than fixed branch-and-bound iteration limits and $i^{qp}_{0}=50$ empirically results in mild tracking error after 30 sample times.

In Figure~\ref{fig:ST_vary_bnb_V1} and Figure~\ref{fig:ST_vary_bnb_V2}, the tracking error is comparable to Figure~\ref{fig:fixed_bnb_ST} for $\underline{i}^{b}=2,5,20$. Unexpectedly, the tracking error for $\underline{i}^{b}=10$ is an order of magnitude less than the rest and is empirically explored further with position plots in Section~\ref{sec:closer_look}.

In Figure~\ref{fig:ST_vary_qp_V1} and Figure~\ref{fig:ST_vary_qp_V2}, there exist $\underline{i}^{qp}$ that track the optimal trajectory better than fixed iterations of $i^{qp}_{0}=1,4,5,7,10$ in Figure~\ref{fig:fixed_qp_ST}. When comparing the branch-and-bound to the quadratic programming iteration limited cases, it is observed that quadratic programming iteration limits affect tracking error significantly more than the branch-and-bound iteration limit. Conceptually this is from the problem formulation of Example~\ref{prob:switching_thrusters}, because branch-and-bound iteration limit only affects the switching thruster, i.e., both thrusters may be on at the same time for some of the MPC horizon, while a quadratic programming iteration limit leads to possible constraint violation of all the constraints. 

Finally, tracking error for $V_{obj}$ and $V_{feas}$ are comparable for branch-and-bound cases in Figure~\ref{fig:ST_vary_bnb_V1} and Figure~\ref{fig:ST_vary_bnb_V2} and empirically have mildly different tracking errors for quadratic programming iteration limits in Figure~\ref{fig:ST_vary_qp_V1} and Figure~\ref{fig:ST_vary_qp_V2}.

\begin{figure}[!ht]
\centering
\begin{subfigure}{0.4\textwidth}
    \includegraphics[width=\textwidth]{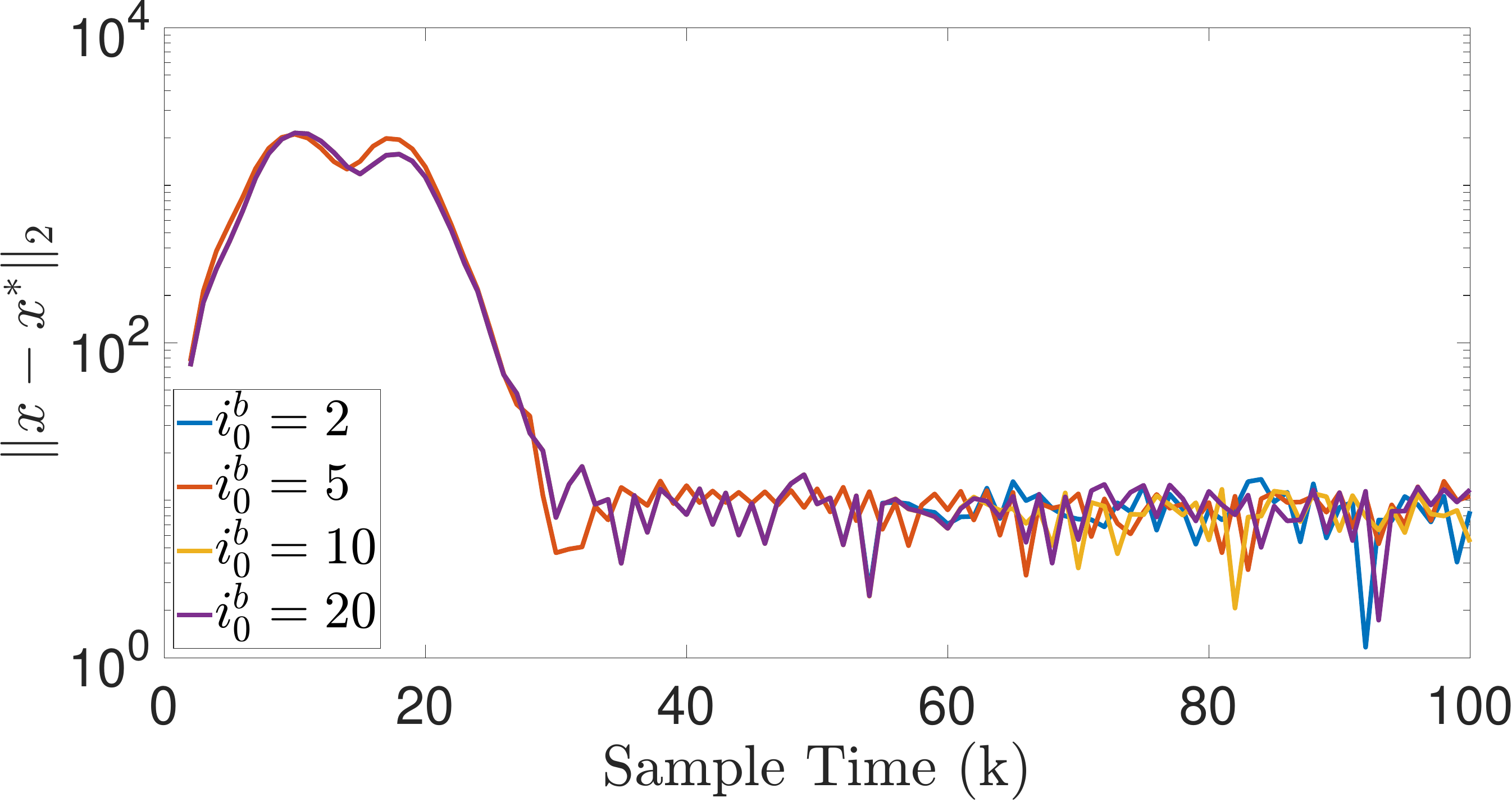}
    \caption{Branch-and-Bound Iteration Limits}
    \label{fig:fixed_bnb_ST}
\end{subfigure}
\hspace{10pt}
\begin{subfigure}{0.4\textwidth}
    \includegraphics[width=\textwidth]{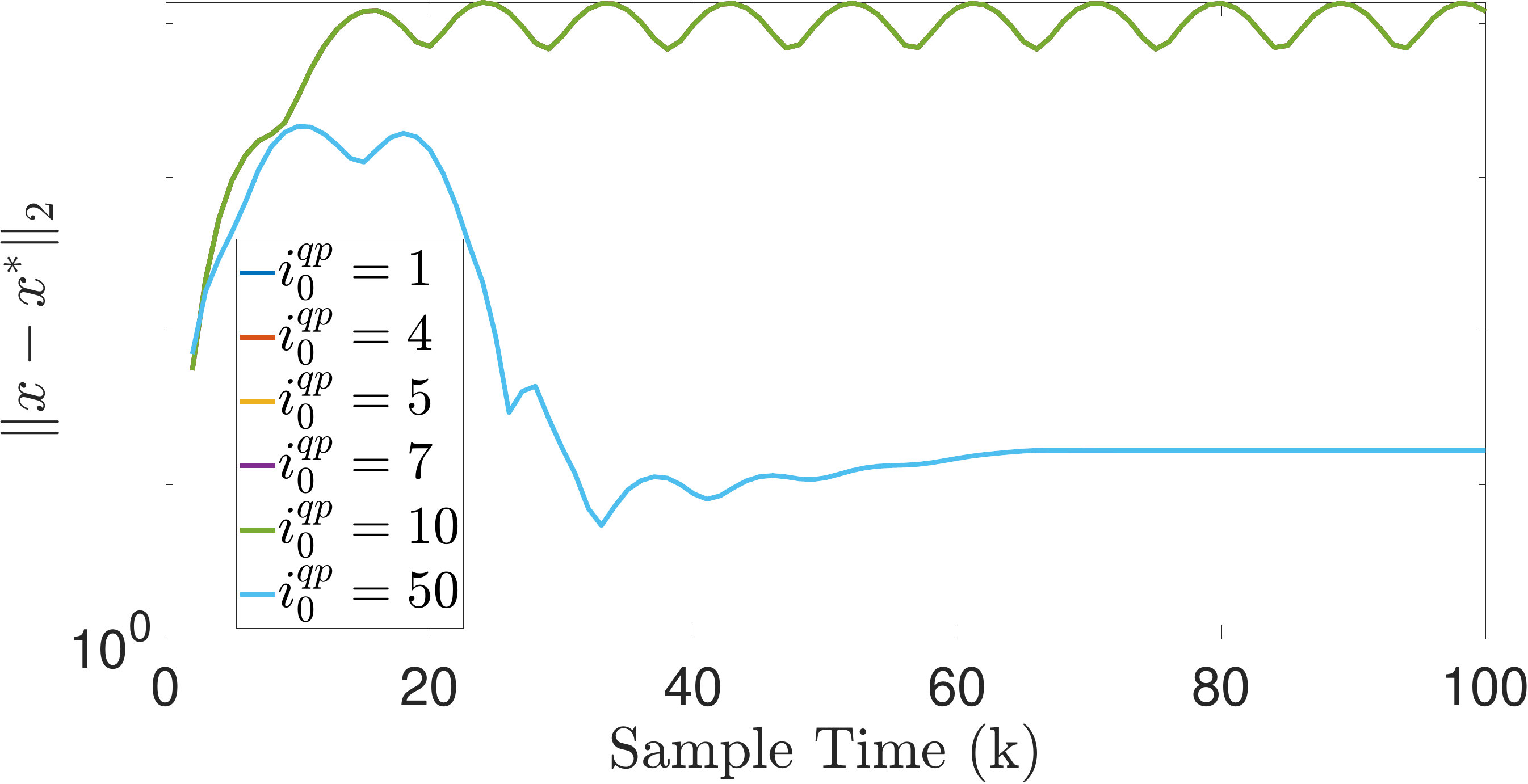}
    \caption{QP Iteration Limits}
     \label{fig:fixed_qp_ST}
\end{subfigure}
\caption{Fixed Iteration Limits for Switching Thruster (Example~\ref{prob:switching_thrusters}).}
% \label{fig:figures}
\end{figure}

\begin{figure}[!ht]
\centering
\begin{subfigure}{0.4\textwidth}
    \includegraphics[width=\textwidth]{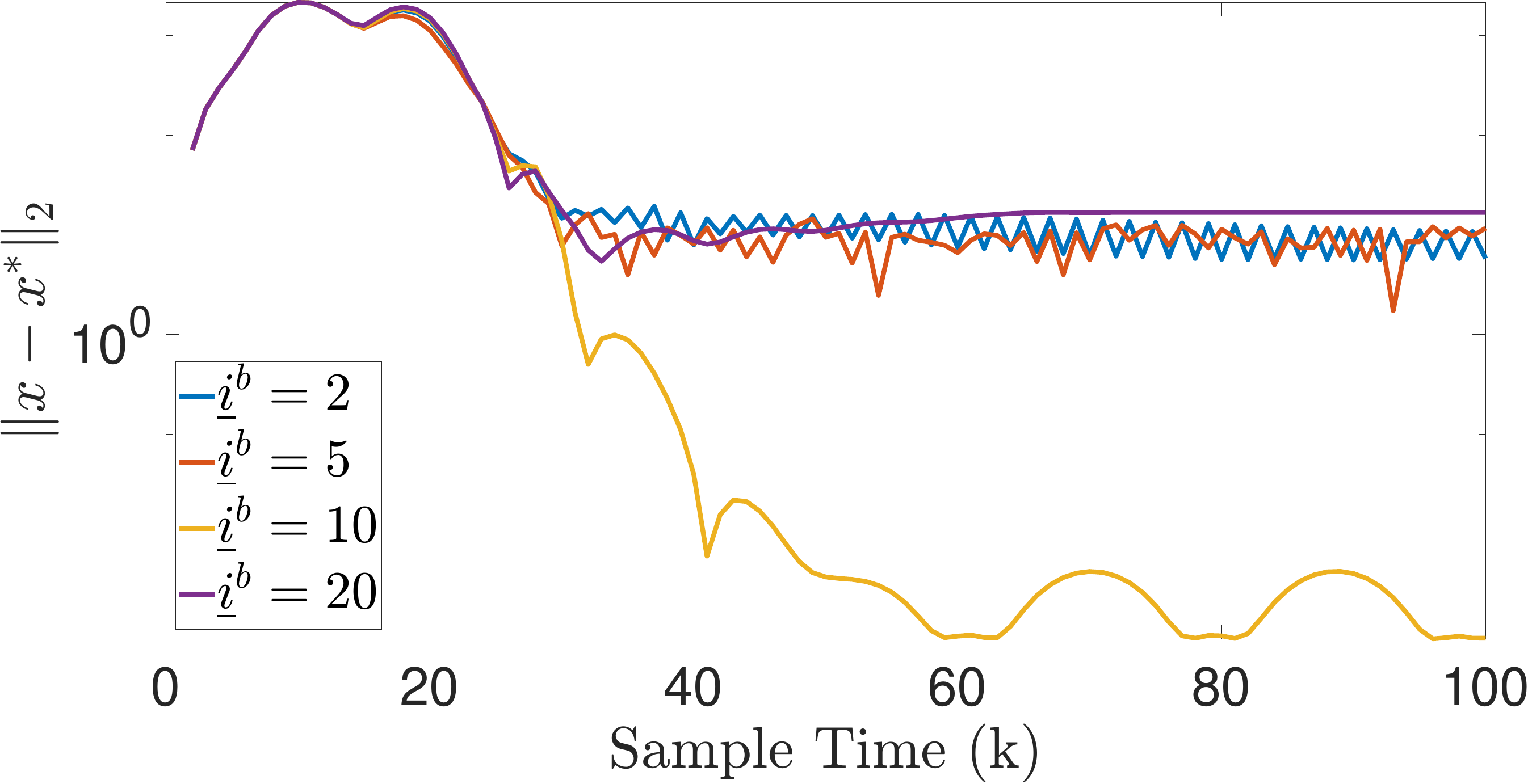}
    \caption{$V_{obj}$ }
    \label{fig:ST_vary_bnb_V1}
\end{subfigure}
\hspace{10pt}
\begin{subfigure}{0.4\textwidth}
    \includegraphics[width=\textwidth]{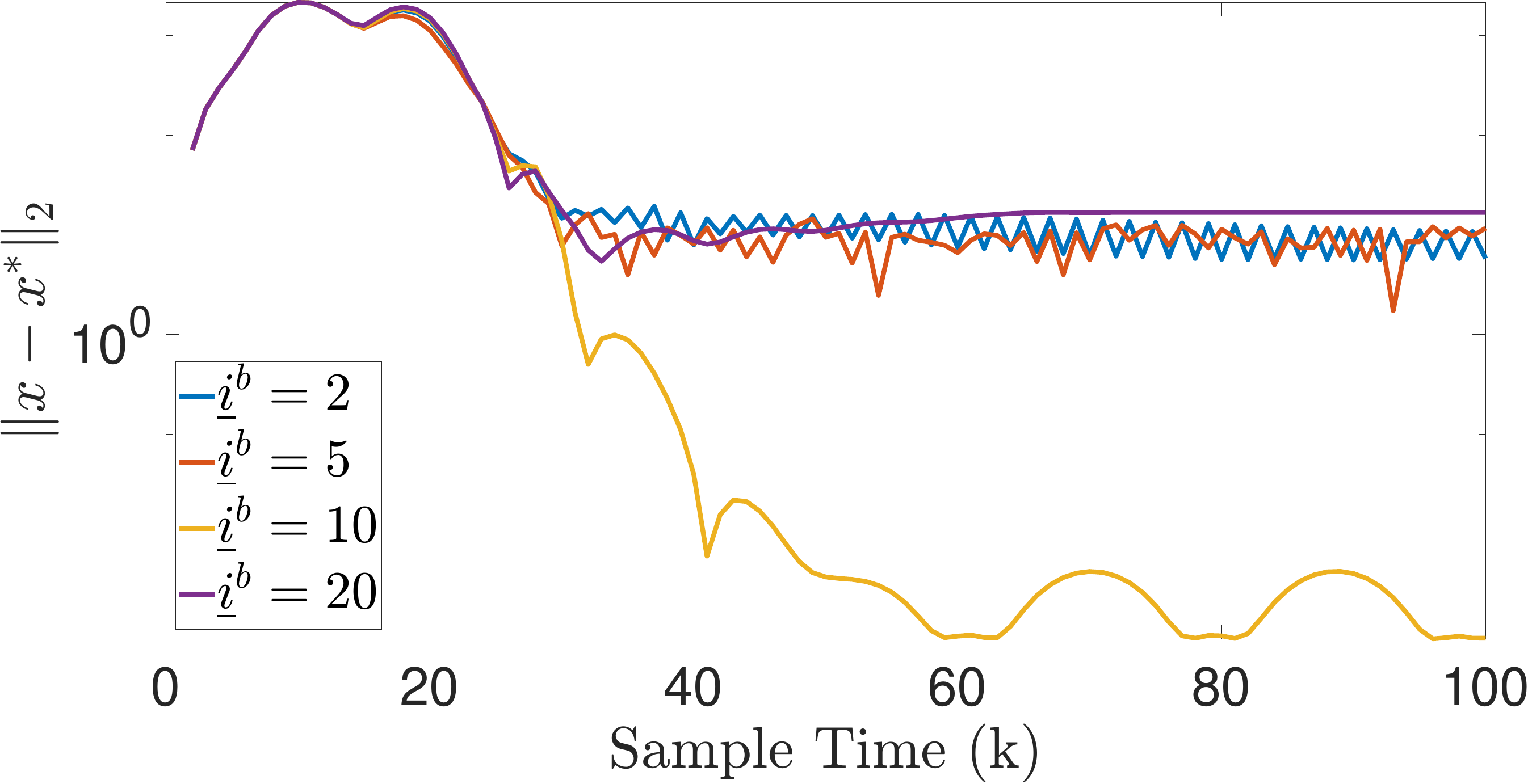}
    \caption{$V_{feas}$ }
    \label{fig:ST_vary_bnb_V2}
\end{subfigure}
\caption{Varying Branch-and-Bound Iteration Limits for Switching Thruster (Example~\ref{prob:switching_thrusters}).}
% \label{fig:figures}
\end{figure}

\begin{figure}[!ht]
\centering
\begin{subfigure}{0.4\textwidth}
    \includegraphics[width=\textwidth]{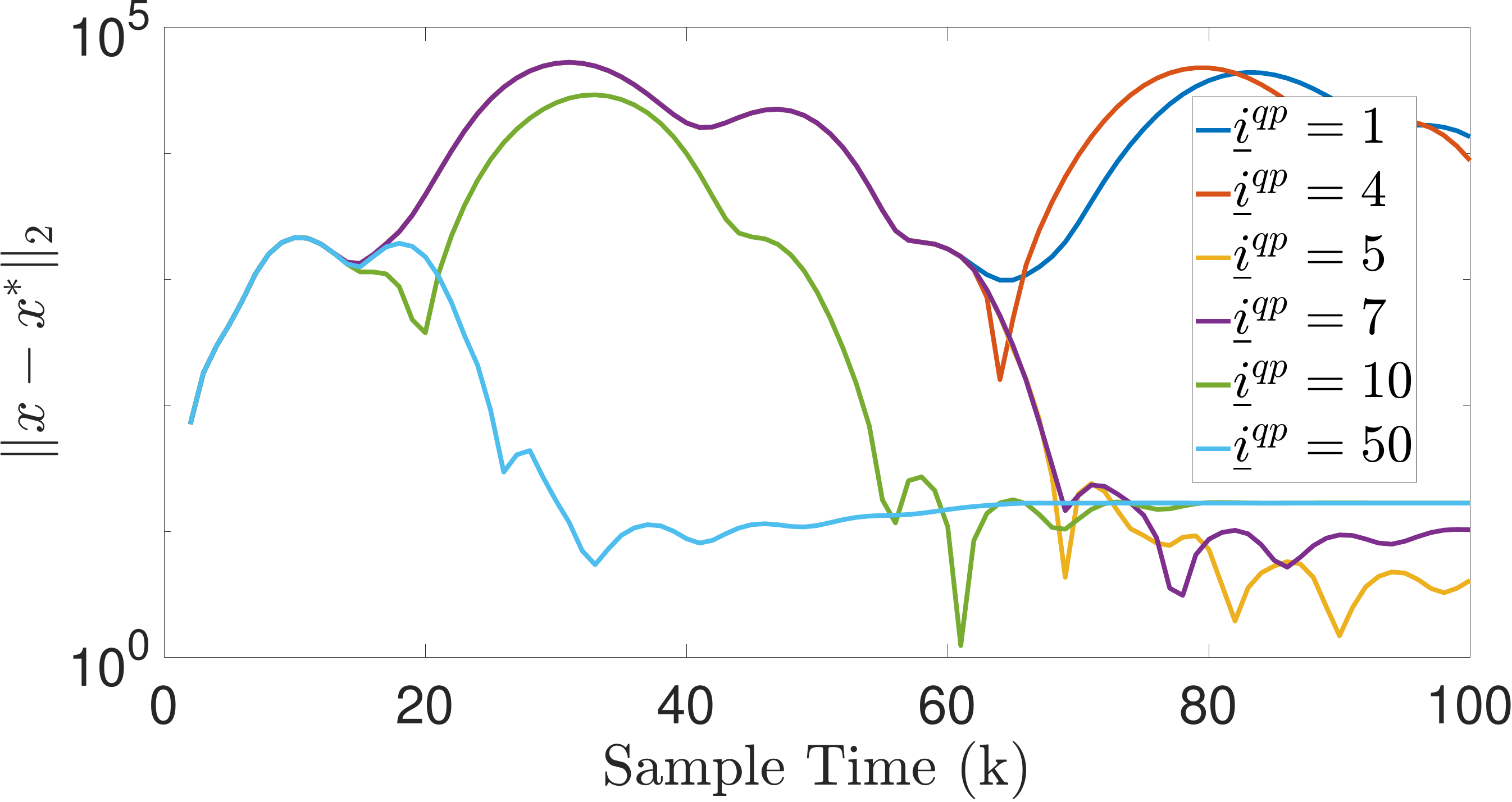}
    \caption{$V_{obj}$ }
    \label{fig:ST_vary_qp_V1}
\end{subfigure}
\hspace{10pt}
\begin{subfigure}{0.4\textwidth}
    \includegraphics[width=\textwidth]{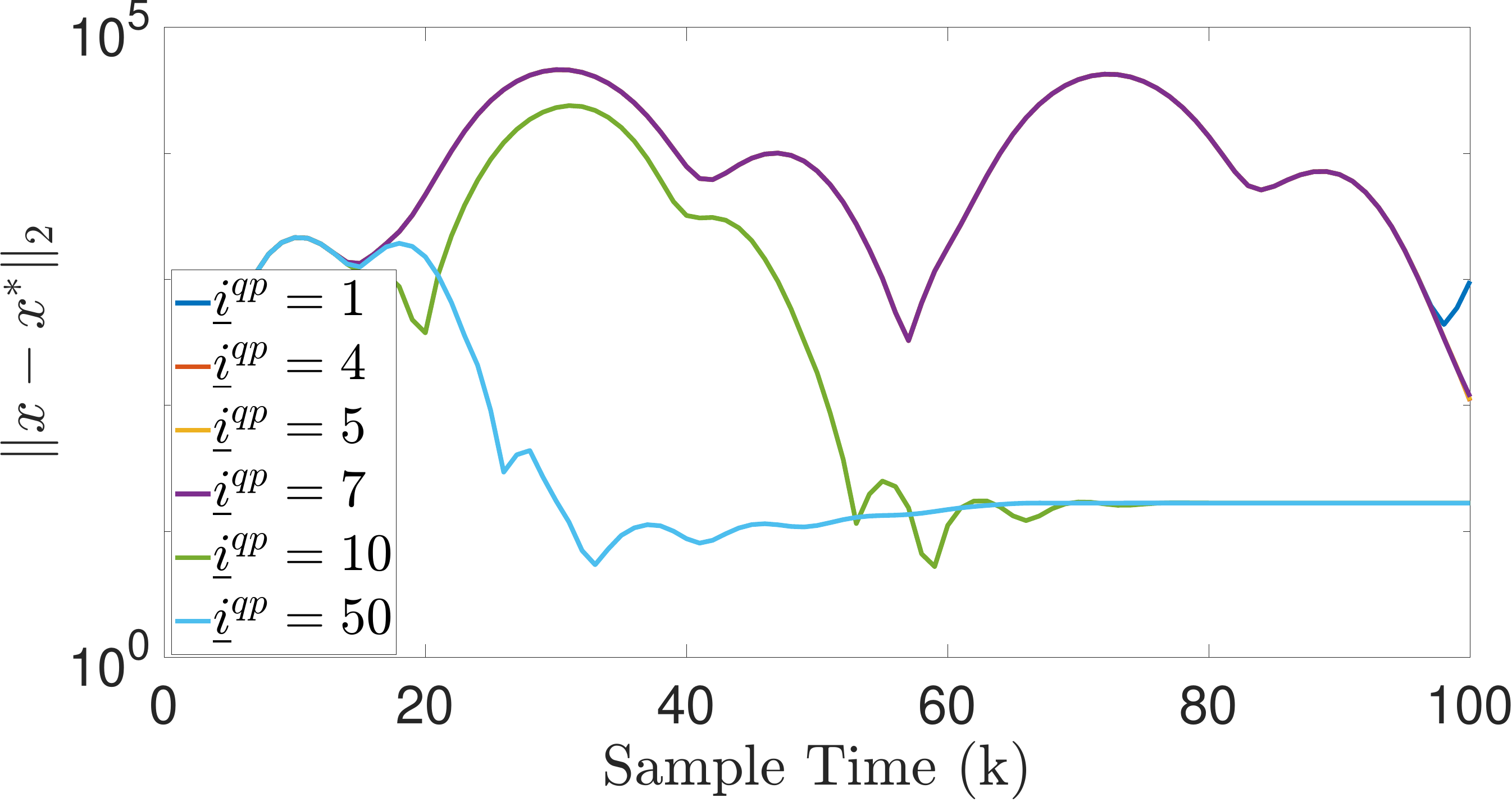}
    \caption{$V_{feas}$ }
    \label{fig:ST_vary_qp_V2}
\end{subfigure}
\caption{Varying QP Iteration Limits for Switching Thruster (Example~\ref{prob:switching_thrusters}).}
% \label{fig:figures}
\end{figure}

\subsubsection{Results for Minimum Thrust}

In Figure~\ref{fig:fixed_bnb_DONUT}, fixed branch-and-bound iteration limit of $i^{b}_{0}=10$ and $i^{b}_{0}=20$ have mild tracking error after 30 sample times. In Figure~\ref{fig:fixed_qp_DONUT}, a fixed QP iteration limit of $i^{qp}_{0}=50$ results in mild tracking after 30 sample times, while the rest have large tracking errors. At a high level, this makes sense because without sufficient branch-and-bound iteration limit the minimum thrust switching model is less accurate, i.e., the minimum thrust constraints may not be strictly satisfied for the whole MPC horizon, but without sufficient quadratic programming iteration limits there is constraint violation of all constraints.

In Figure~\ref{fig:DONUT_vary_bnb_V1} and Figure~\ref{fig:DONUT_vary_bnb_V2}, tracking error after 30 sample times is comparable to Figure~\ref{fig:fixed_bnb_DONUT} for branch-and-bound iteration limit of $i^{b}_{0}=10$ and $i^{b}_{0}=20$. As expected, the tracking error for a branch-and-bound iteration limit of $\underline{i}^{b}=2$ and $\underline{i}^{b}=5$ track the optimal solution better than the fixed iteration limit case. This is due to a sufficiently large iteration limit for the complex part of the trajectory and is empirically explained further in Section~\ref{sec:closer_look}.

As expected in Figure~\ref{fig:DONUT_vary_qp_V1} and Figure~\ref{fig:DONUT_vary_qp_V2}, $\underline{i}^{qp}$ iteration limits track the optimal trajectory better than fixed iteration limits of 1 through 10 in Figure~\ref{fig:fixed_qp_ST}. Finally, $V_{obj}$ and $V_{feas}$ empirically have comparable tracking error for branch-and-bound cases and have mildly different tracking error for the quadratic programming case.

Example~\ref{prob:min_thrust_MIQP} has an order of magnitude less tracking error on average for the uniting control law QP iteration limit compared to Example~\ref{prob:switching_thrusters}. This is likely due to the optimization problem specific formulation differences, i.e., Example~\ref{prob:switching_thrusters} has more optimization variables and a more complex controlled dynamical system.

\begin{figure}[!ht]
\centering
\begin{subfigure}{0.4\textwidth}
    \includegraphics[width=\textwidth]{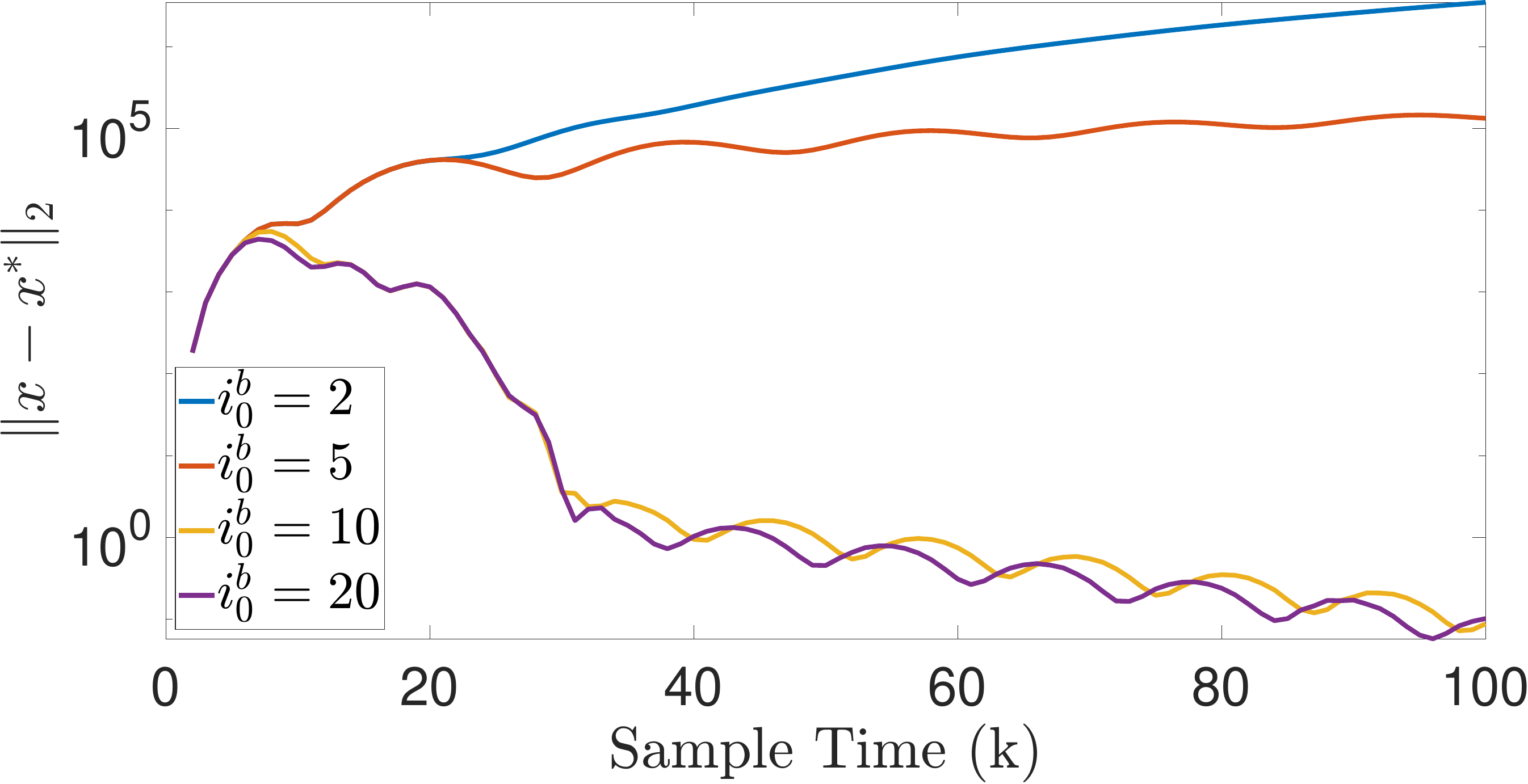}
    \caption{Branch-and-Bound Iteration Limits}
    \label{fig:fixed_bnb_DONUT}
\end{subfigure}
\hspace{10pt}
\begin{subfigure}{0.4\textwidth}
    \includegraphics[width=\textwidth]{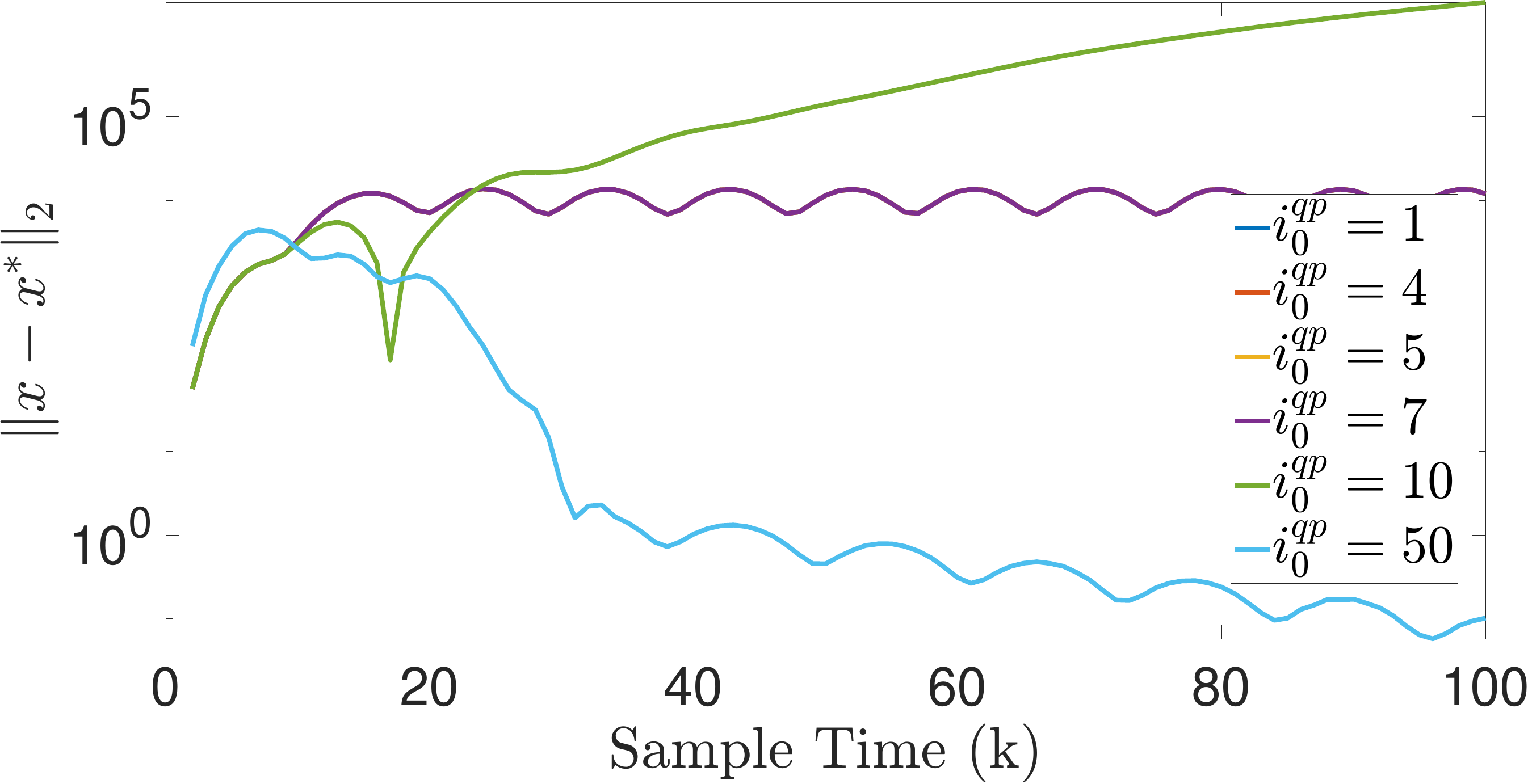}
    \caption{QP Iteration Limits}
    \label{fig:fixed_qp_DONUT}
\end{subfigure}
\caption{Fixed Iteration Limits for Minimum Thrust (Example~\ref{prob:min_thrust_MIQP}).}
% \label{fig:figures}
\end{figure}

\begin{figure}[!ht]
\centering
\begin{subfigure}{0.4\textwidth}
    \includegraphics[width=\textwidth]{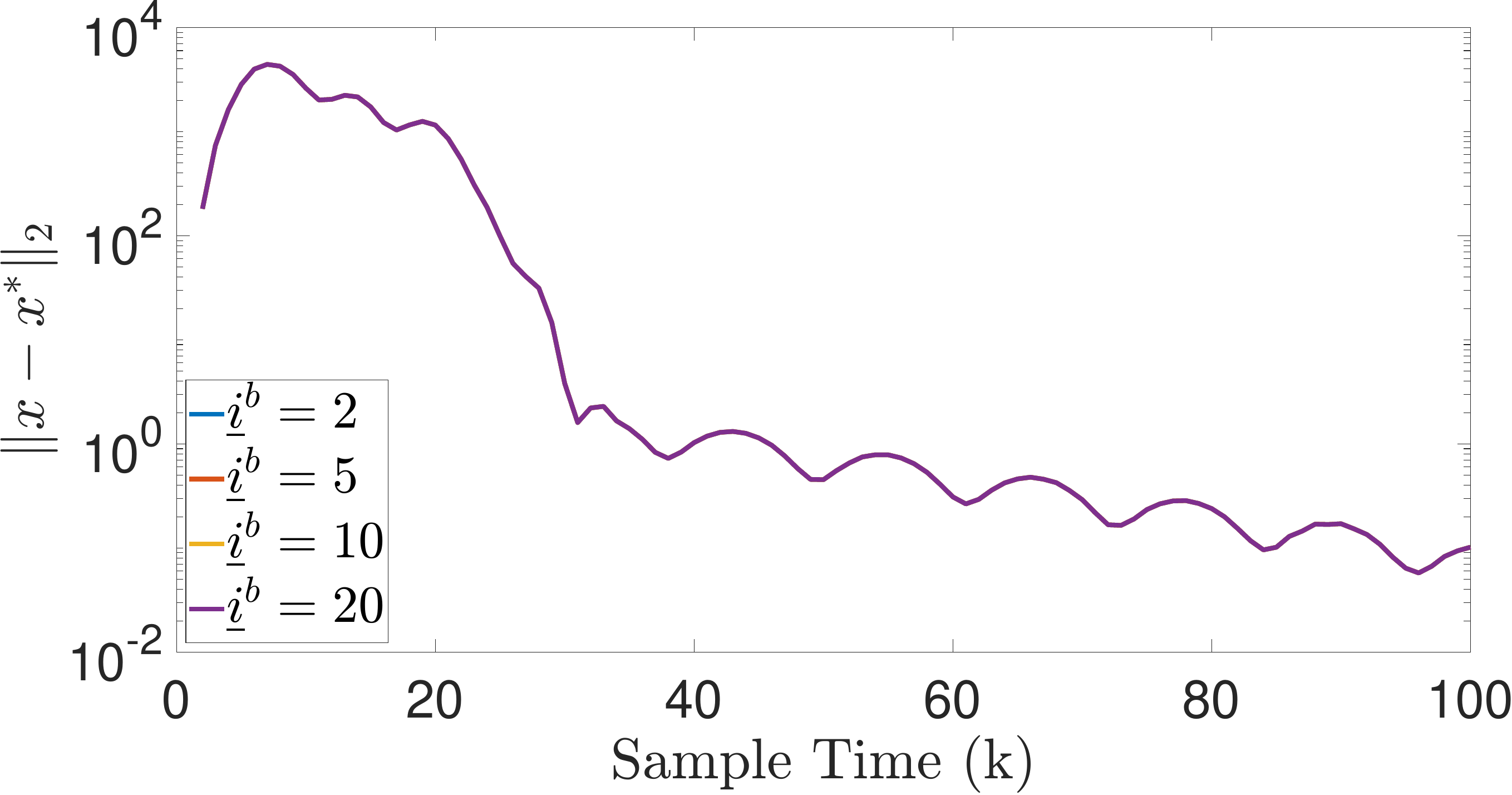}
    \caption{$V_{obj}$}
    \label{fig:DONUT_vary_bnb_V1}
\end{subfigure}
\hspace{10pt}
\begin{subfigure}{0.4\textwidth}
    \includegraphics[width=\textwidth]{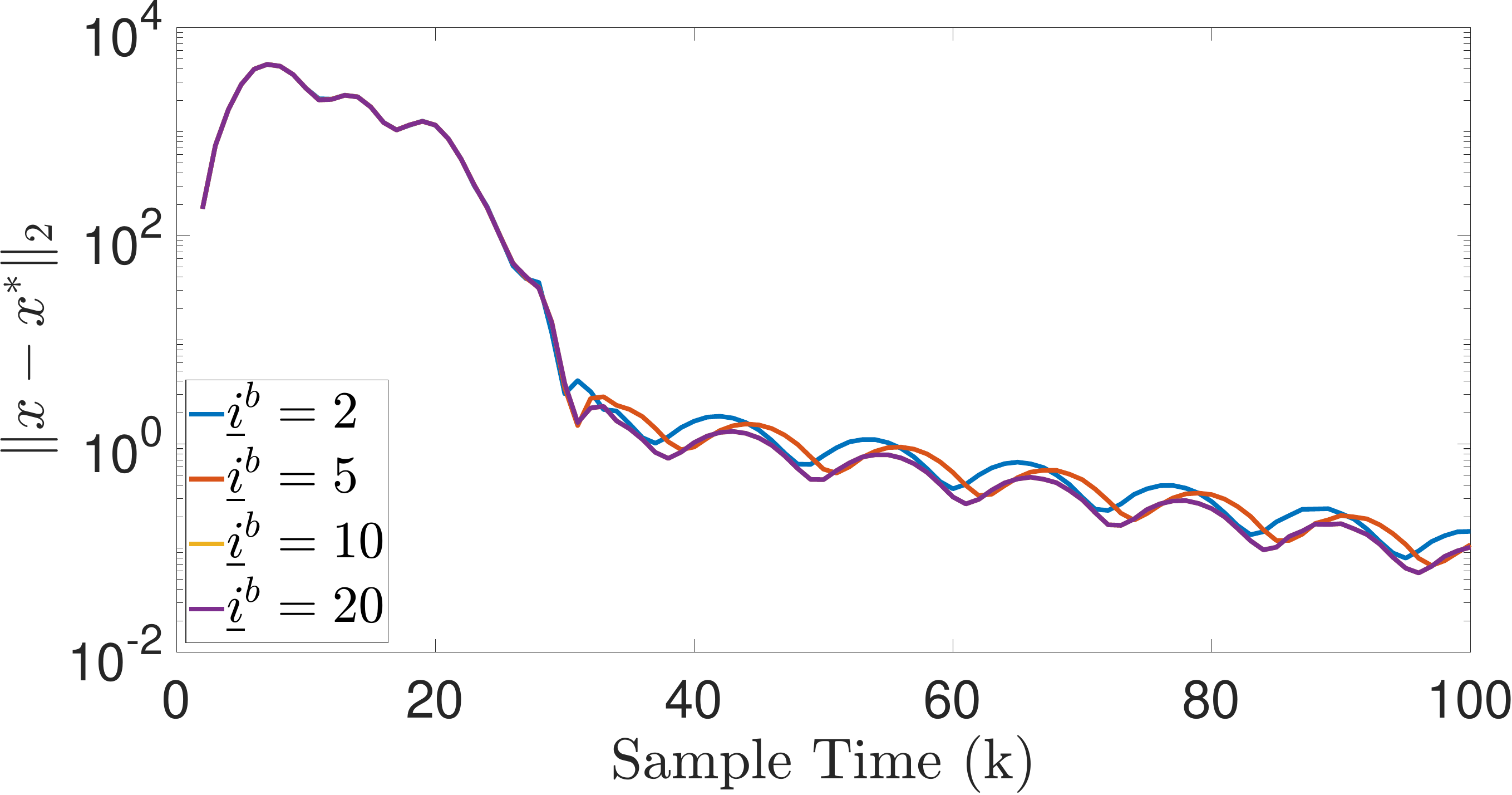}
    \caption{$V_{feas}$}
    \label{fig:DONUT_vary_bnb_V2}
\end{subfigure}
\caption{Varying Branch-and-Bound Iteration Limits for Minimum Thrust (Example~\ref{prob:min_thrust_MIQP}).}
% \label{fig:figures}
\end{figure}

\begin{figure}[!ht]
\centering
\begin{subfigure}{0.4\textwidth}
    \includegraphics[width=\textwidth]{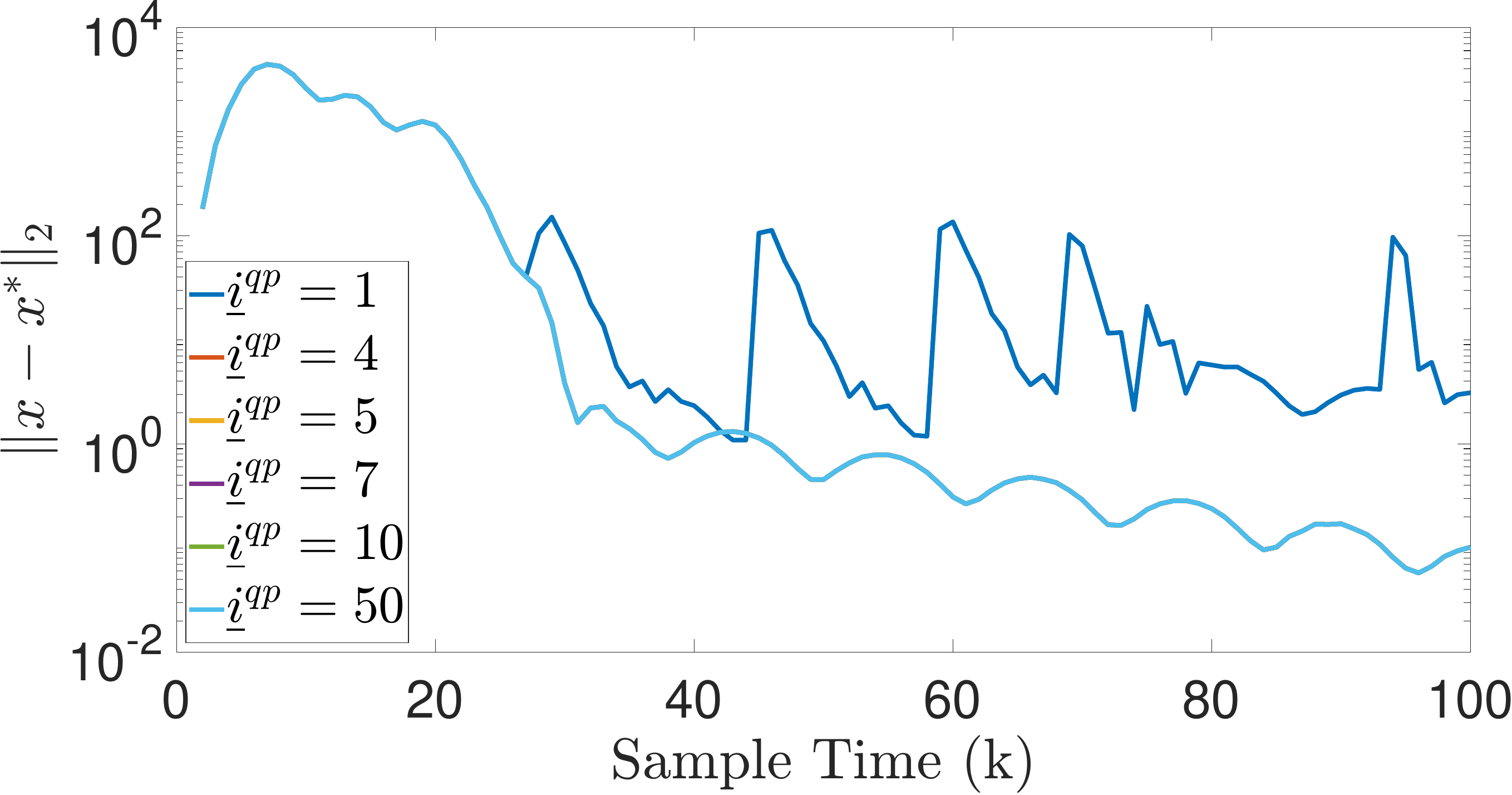}
    \caption{$V_{obj}$}
    \label{fig:DONUT_vary_qp_V1}
\end{subfigure}
\hspace{10pt}
\begin{subfigure}{0.4\textwidth}
    \includegraphics[width=\textwidth]{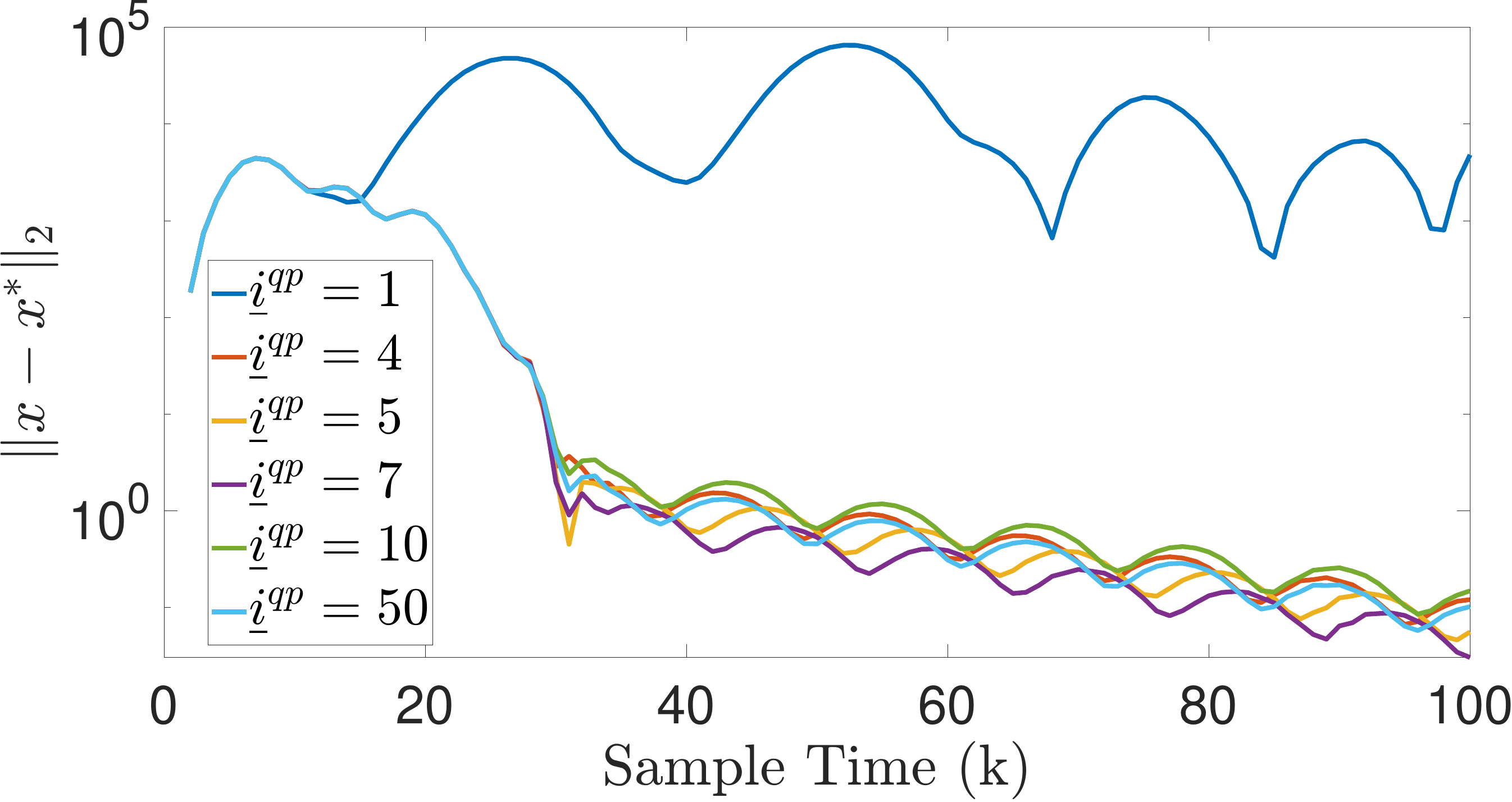}
    \caption{$V_{feas}$}
    \label{fig:DONUT_vary_qp_V2}
\end{subfigure}
\caption{Varying QP Iteration Limits for Minimum Thrust (Example~\ref{prob:min_thrust_MIQP}).}
% \label{fig:figures}
\end{figure}

\subsection{A Closer Look at Algorithm~\ref{alg:uniting_control} Near the Terminal Set}\label{sec:closer_look}
Comparing the tracking error between the optimal and suboptimal trajectories is insufficient to fully understand the behavior. In this section, the trajectories are empirically explored to understand what changes before and after the 30 sample time mark to sharply decrease the tracking error. Since $V_{obj}$ and $V_{feas}$ are comparable in performance, this section explores the trajectories for only $V_{feas}.$

Figure~\ref{fig:ST_pos_bnb} provides context for the difference between varying branch-and-bound iteration limits observed after 30 sample times in Figure~\ref{fig:ST_vary_bnb_V2}, i.e., $\underline{i}^{b}=10$ has significantly less chatter approaching the origin than the branch-and-bound iteration limit of  $\underline{i}^{b}=2$ and  $\underline{i}^{b}=5$. Due to the dynamics and numerics of switching between an electric and chemical thruster, the branch-and-bound iteration limit of  $\underline{i}^{b}=10$ empirically reaches a stationary point near the optimal solution while the branch-and-bound iteration limit of $\underline{i}^{b}=20$ has a larger deviation from the optimal solution. For Example~\ref{prob:min_thrust_MIQP}, Figure~\ref{fig:DONUT_pos_bnb} illustrates a smaller trajectory chatter approaching the origin with final Euclidean norm error relative to the origin of $0.1443,~0.1069,~0.1017,~0.1017,$ and $9.0609\cdot 10^{-4}$ for $\underline{i}^{b}=2,~\underline{i}^{b}=5,~\underline{i}^{b}=10,\underline{i}^{b}=20,$ and optimal at the end of the simulation, respectively . Figure~\ref{fig:ST_bnb_vary_pos_part2}, in contrast, has a larger trajectory chatter approaching the origin with a final Euclidean norm error relative to the origin of $5.8054,~11.6990,~3.3998\cdot10^{-7},~16.8119$ and $9.0609\cdot10^{-4}$ for $\underline{i}^{b}=2,~\underline{i}^{b}=5,~\underline{i}^{b}=10,\underline{i}^{b}=20,$  and optimal at the end of the simulation, respectively.

For branch-and-bound iteration limits, the trajectory from Algorithm~\ref{alg:uniting_control} is interpretable since it is solved with Algorithm~\ref{alg:bnb} and comparable to the optimal solution. After 30 sample times, Figure~\ref{fig:ST_pos_bnb} and Figure~\ref{fig:DONUT_pos_bnb} highlight the practical importance of a numerical optimizer with tight numerical accuracy when the spacecraft needs to exactly reach the terminal point. The chattering behavior near the origin in Figure~\ref{fig:ST_bnb_vary_pos_part2} and Figure~\ref{fig:DONUT_bnb_vary_pos_part2} illustrate that Assumption~\ref{assum:asym_control}.2 does not hold in practice for sufficiently user-desired small iteration limits and $\varepsilon_{0}$ is sufficiently small. This highlights a decision for practitioners: what is a sufficient numerical accuracy, $\epsilon_{0}$, and sufficient iteration limit to achieve the desired $\epsilon_{0}$? Additionally,  Figure~\ref{fig:ST_bnb_vary_pos_part2} and Figure~\ref{fig:DONUT_bnb_vary_pos_part2} empirically verify the uniting control law constructed via a hybrid model in Section~\ref{sec:hybrid_theory} is an effective method to achieve asymptotic stability and reduce computational usage via iteration limits.

\begin{figure}[!ht]
\centering
\begin{subfigure}{0.4\textwidth}
    \includegraphics[width=\textwidth]{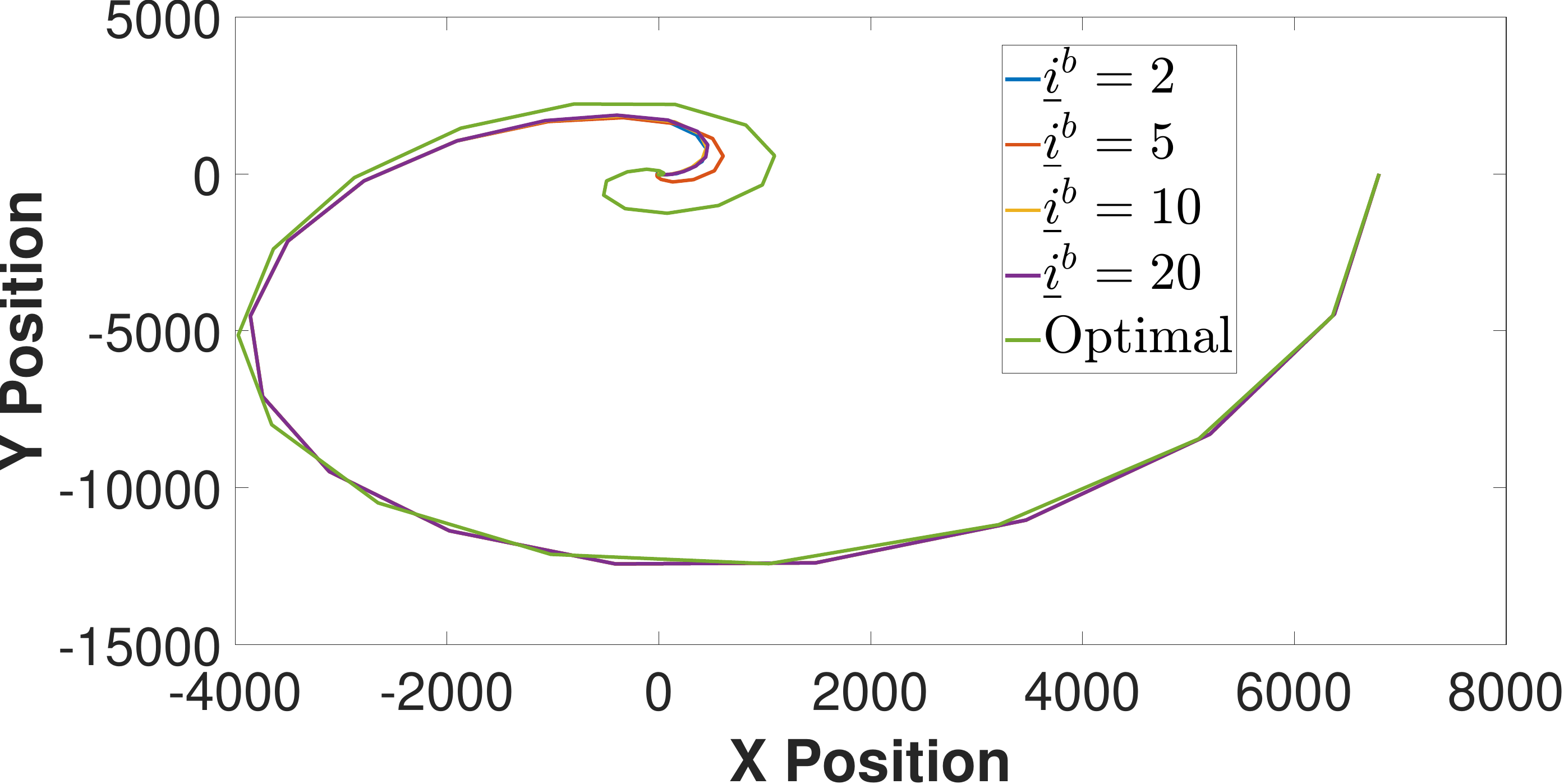}
    \caption{Position for first 30 sample times.}
    \label{fig:ST_bnb_vary_pos_part1}
\end{subfigure}
\hspace{10pt}
\begin{subfigure}{0.4\textwidth}
    \includegraphics[width=\textwidth]{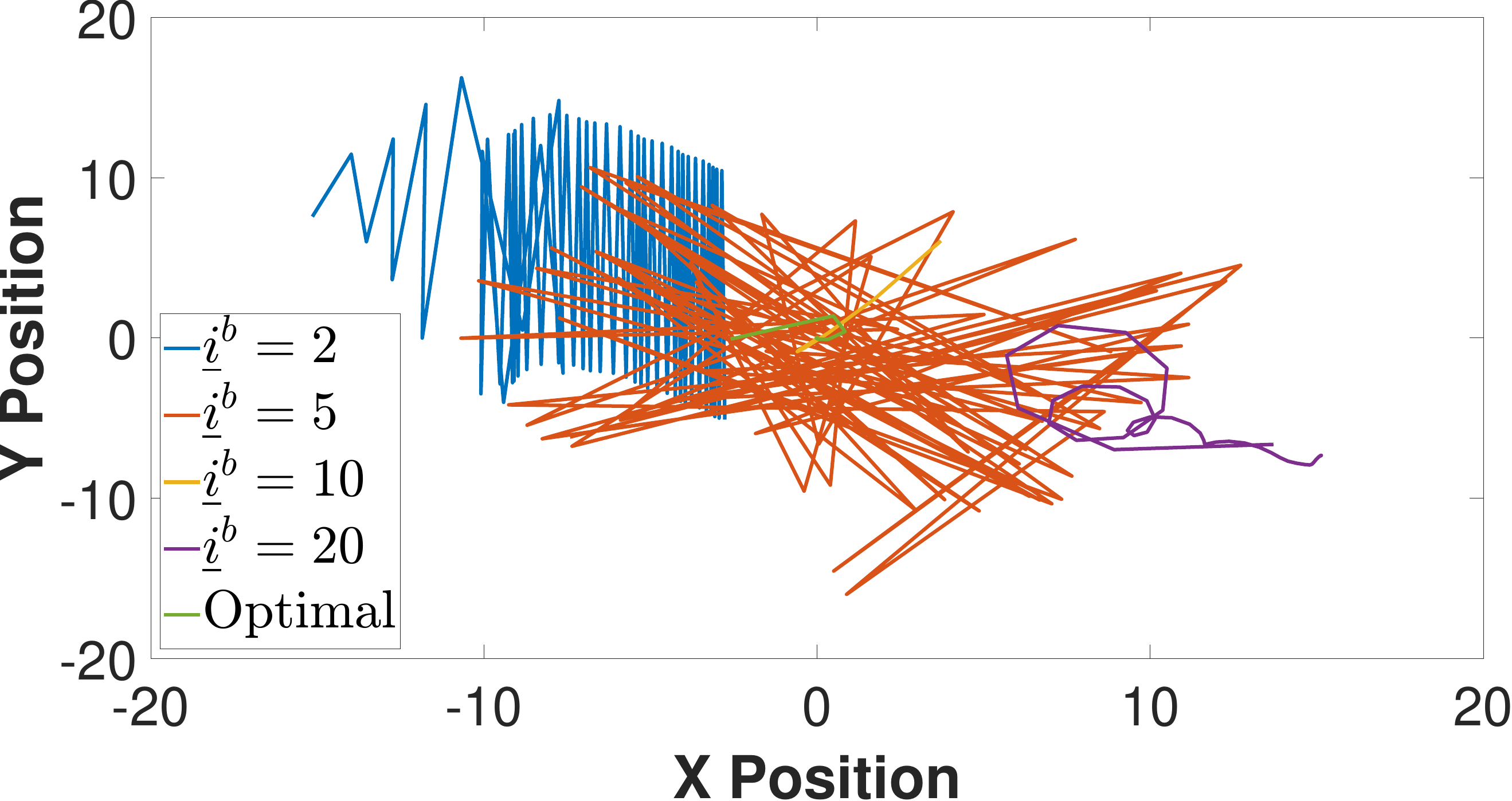}
    \caption{Position after 30 sample times to end.}
    \label{fig:ST_bnb_vary_pos_part2}
\end{subfigure}
\caption{Varying Branch-and-Bound Iteration Limits for Switching Thrust (Example~\ref{prob:switching_thrusters}) with $V_{feas}$.}
\label{fig:ST_pos_bnb}
\end{figure}

\begin{figure}[!ht]
\centering
\begin{subfigure}{0.4\textwidth}
    \includegraphics[width=\textwidth]{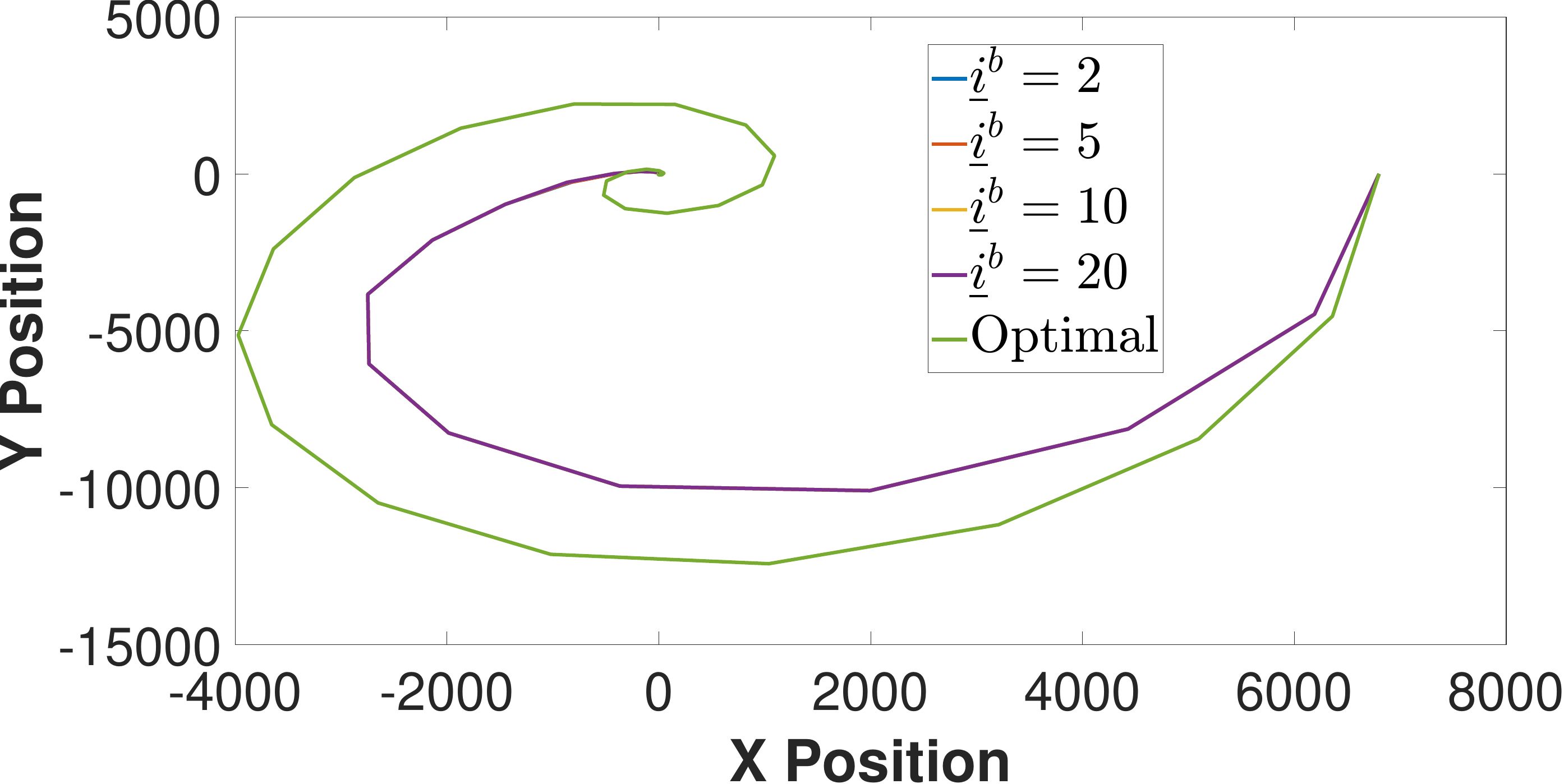}
    \caption{Position for first 30 sample times.}
    % \label{fig:first}
\end{subfigure}
\hspace{10pt}
\begin{subfigure}{0.4\textwidth}
    \includegraphics[width=\textwidth]{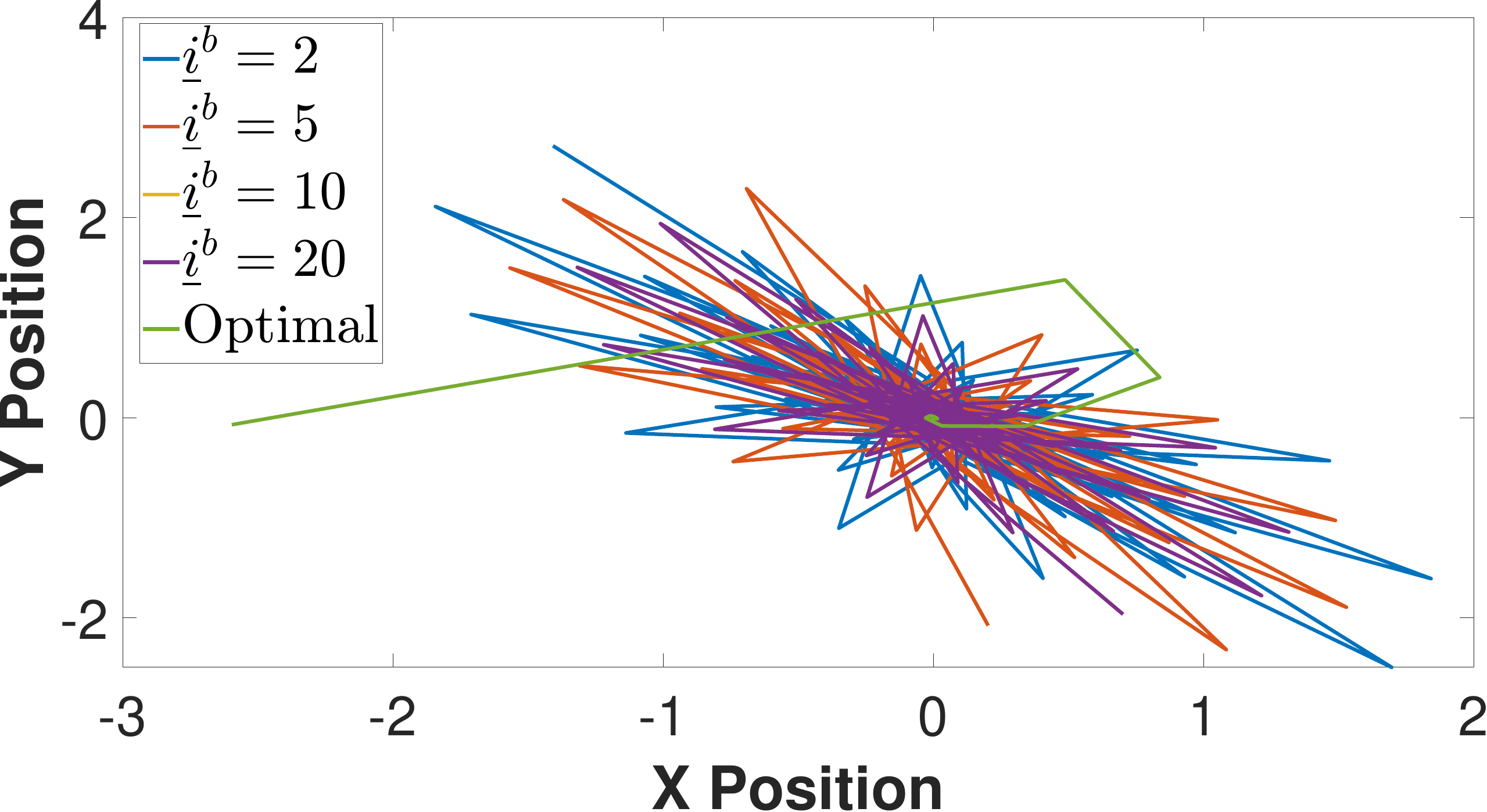}
    \caption{Position after 30 sample times to end.}
    \label{fig:DONUT_bnb_vary_pos_part2}
\end{subfigure}
\caption{Varying Branch-and-Bound Iteration Limits for Minimum Thrust (Example~\ref{prob:min_thrust_MIQP}) with $V_{feas}$.}
\label{fig:DONUT_pos_bnb}
\end{figure}

Figure~\ref{fig:ST_qp_pos} and Figure~\ref{fig:DONUT_qp_pos} provide context for tracking error of the closed loop system before and after 30 sample times, which are more exotic than the branch-and-bound cases. To note, trajectories not visible on Figure~\ref{fig:ST_qp_vary_pos_part2} and Figure~\ref{fig:DONUT_qp_vary_pos_part2}, e.g., an iteration limit of 1 are unstable. This coincides with the computational exploratory work for a nonlinear MPC problem that empirically verifies there are minimum user-desired iteration limits to achieve stability of the closed loop system \cite{behrendt2025time}.

\begin{figure}[!ht]
\centering
\begin{subfigure}{0.4\textwidth}
    \includegraphics[width=\textwidth]{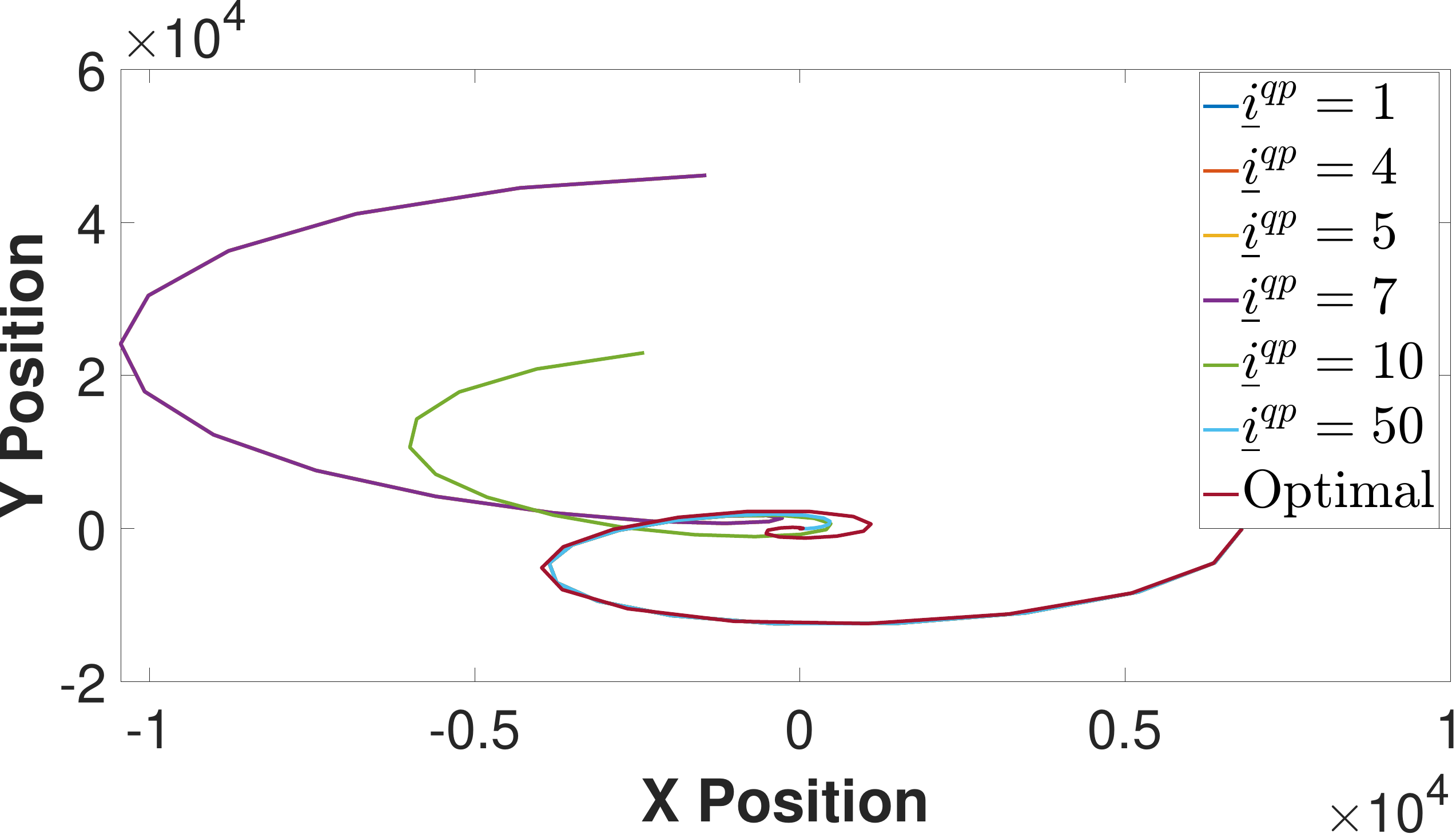}
    \caption{Position for first 30 sample times.}
    % \label{fig:first}
\end{subfigure}
\hspace{10pt}
\begin{subfigure}{0.4\textwidth}
    \includegraphics[width=\textwidth]{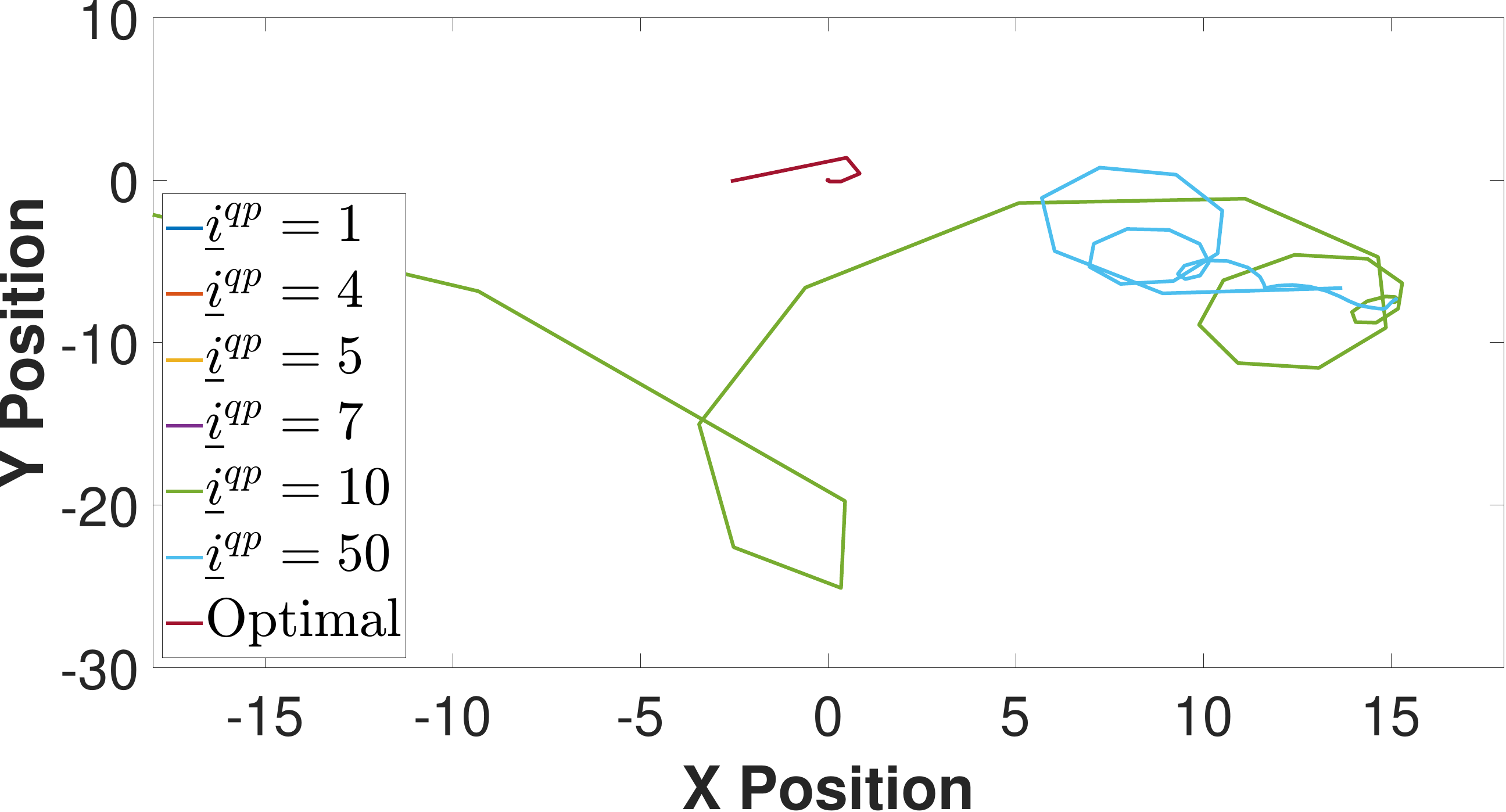}
    \caption{Position after 30 sample times to end.}
    \label{fig:ST_qp_vary_pos_part2}
\end{subfigure}
\caption{Varying Quadratic Programming iterations for Switching Thrust (Example~\ref{prob:switching_thrusters}) with $V_{feas}$.}
\label{fig:ST_qp_pos}
\end{figure}

\begin{figure}[!ht]
\centering
\begin{subfigure}{0.4\textwidth}
    \includegraphics[width=\textwidth]{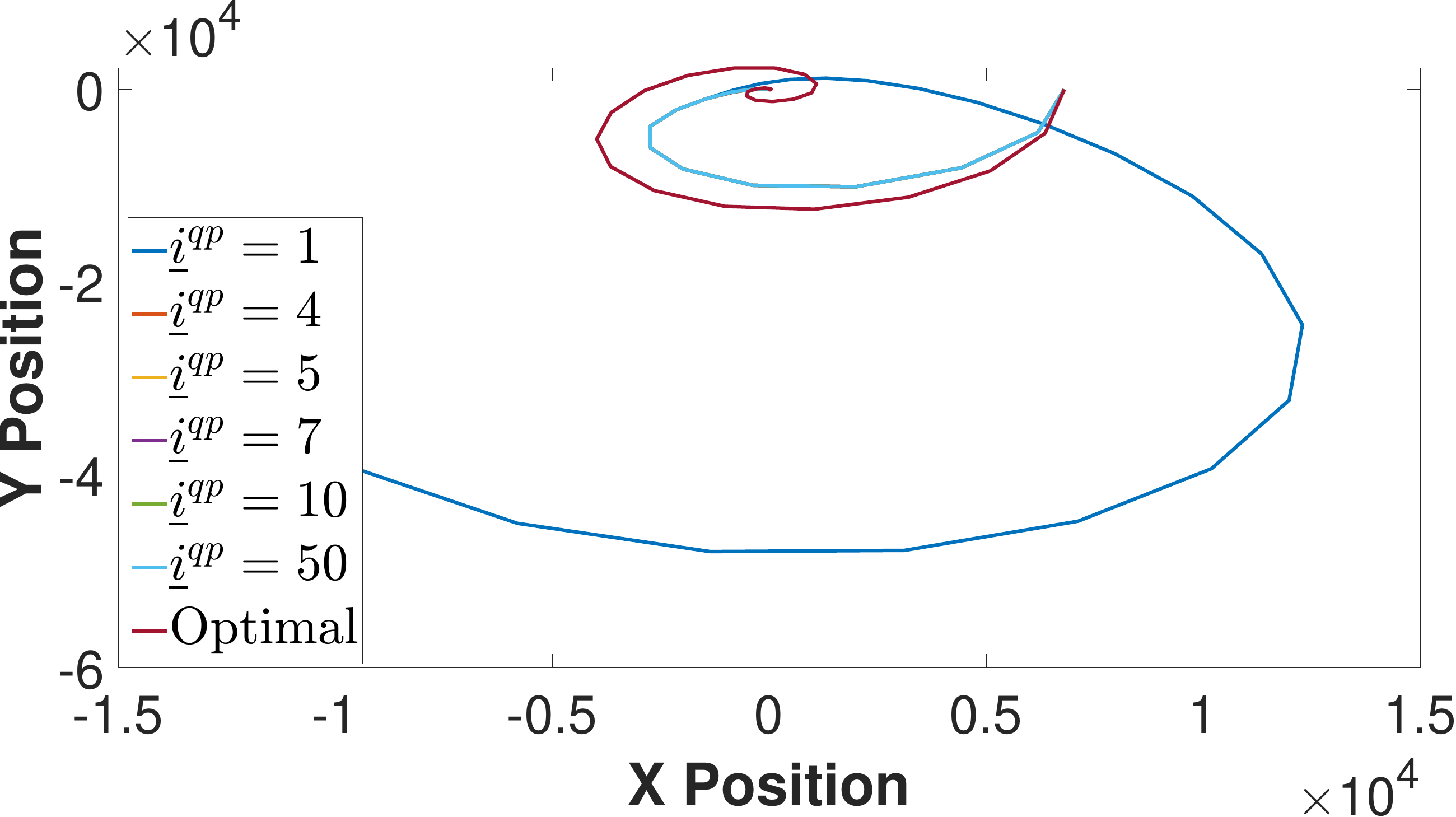}
    \caption{Position for first 30 sample times.}
    % \label{fig:first}
\end{subfigure}
\hspace{10pt}
\begin{subfigure}{0.4\textwidth}
    \includegraphics[width=\textwidth]{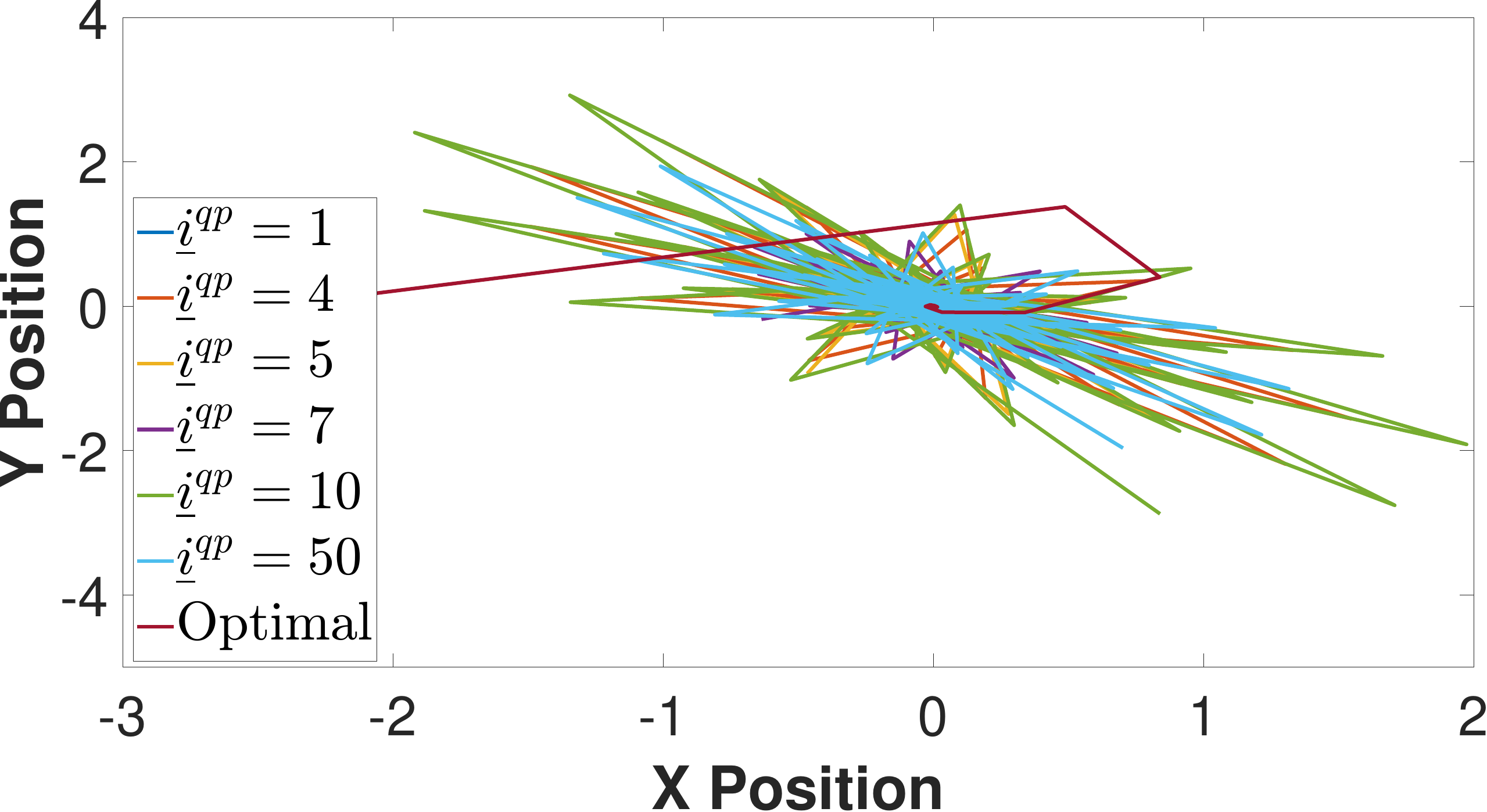}
    \caption{Position after 30 sample times to end.}
    \label{fig:DONUT_qp_vary_pos_part2}
\end{subfigure}
\caption{Varying Quadratic Programming iterations for Minimum Thrust (Example~\ref{prob:min_thrust_MIQP}) with $V_{feas}$.}
\label{fig:DONUT_qp_pos}
\end{figure}

\subsection{Computational Improvements}\label{sec:compute_improvements}
In Section~\ref{sec:closer_look}, by inspection the bulk of challenging trajectory optimization occurs within the first 30 sample time and the rest of the trajectory has numerical challenges. This section compares the average QP iteration limits and branch-and-bound iteration limit for Example~\ref{prob:switching_thrusters} and Example~\ref{prob:min_thrust_MIQP} with initial conditions of $i_{0}^{b}=20$ and $i_{0}^{qp}=100$ for the first 30 sample times. The results presented illustrate that both $V_{obj}$ and $V_{feas}$ are effective measures of perturbation on the set $\mathcal{A}$ and iteration limits effectively reduce the average number of iterations. In Table~\ref{tab:bnb_switching}, the average branch-and-bound iteration is the same for Example~\ref{prob:switching_thrusters} or slightly worse for $V_{feas}$ compared to $V_{obj}$. However in Table~\ref{tab:bnb_donut}, $V_{feas}$ results in a lower average branch-and-bound iteration for all cases. In Table~\ref{tab:qp_switching}, the average quadratic programming iteration is comparable for Example~\ref{prob:switching_thrusters} between $V_{obj}$ and $V_{feas}$. In Table~\ref{tab:qp_donut}, the average quadratic programming iteration for $V_{feas}$ is significantly lower than $V_{obj}$.

\begin{table}[!ht]
    \centering
    \begin{tabular}{|c|c|c|c|c|}
    \hline
    BnB Iters ($i^{b}$)&$i^{b}$ &$i^{b}$&$i^{b}$&$i^{b}$\\
    \hline 
        Lyapunov Function & 2&  5&  10&20\\
        \hline
         $V_{obj}$&  9.2&  11&  14& 20\\
         \hline
         $V_{feas}$&  9.2&  11&  14& 20\\
         \hline
    \end{tabular}
    \caption{Average Branch-and-Bound Iteration Limits for Example~\ref{prob:switching_thrusters} in the first 30 sample times.}
    \label{tab:bnb_switching}
\end{table}

\begin{table}[!ht]
    \centering
    \begin{tabular}{|c|c|c|c|c|c|}
    \hline
     BnB Iters ($i^{b}$)&$i^{b}$ &$i^{b}$&$i^{b}$&$i^{b}$\\
    \hline
        Lyapunov Function & 2&  5& 10&20 \\
        \hline
         $V_{obj}$&  17.6&  18&  18.7& 20\\
         \hline
         $V_{feas}$&  7.4&  9.5&  13& 20\\
         \hline
    \end{tabular}
    \caption{Average Branch-and-Bound Iteration Limits for Example~\ref{prob:min_thrust_MIQP} in the first 30 sample times.}
    \label{tab:bnb_donut}
\end{table}

\begin{table}[!ht]
    \centering
    \begin{tabular}{|c|c|c|c|c|c|c|}
    \hline
     QP Iters ($i^{qp}$)&$i^{qp}$ &$i^{qp}$&$i^{qp}$&$i^{qp}$&$i^{qp}$ & $i^{qp}$\\
    \hline
        Lyapunov Function & 1&  4&  5&7& 10& 50\\
        \hline
         $V_{obj}$&  70.3&  71.2&  71.5& 72.1& 58.0& 70\\
         \hline
         $V_{feas}$&  73.6&  74.4&  74.7& 75.2& 67.0& 70 \\
         \hline
    \end{tabular}
    \caption{Average QP Iteration Limits for Example~\ref{prob:switching_thrusters} in the first 30 sample times.}
   \label{tab:qp_switching}
\end{table}

\begin{table}[!ht]
    \centering
    \begin{tabular}{|c|c|c|c|c|c|c|}
    \hline
      QP Iters ($i^{qp}$)&$i^{qp}$ &$i^{qp}$&$i^{qp}$&$i^{qp}$&$i^{qp}$ & $i^{qp}$\\
    \hline
        Lyapunov Function & 1&  4&  5&7& 10&  50\\
        \hline
         $V_{obj}$&  96.7&  87.2&  87.3& 87.6& 88.0& 93.3\\
         \hline
         $V_{feas}$&  70.3&  32.8&  33.5& 34.9& 37.0 &65.0 \\
         \hline
    \end{tabular}
    \caption{Average QP Iteration Limits for Example~\ref{prob:min_thrust_MIQP} in the first 30 sample times.}
    \label{tab:qp_donut}
\end{table}

\subsection{Tuning Uniting Control Law Terms} \label{sec:tuning}

From Section~\ref{sec:compute_improvements}, no uniting control case has an average iteration less than the fixed iteration limits. However, empirically there exist fixed iteration limit cases that lead to lower iterations than the uniting control law. Empirically, this suggests the uniting control law is conservative and further tuning of the constants $c_{p,0}$ and $~c_{p,1}$  is possible. 

When stability and robustness is crucial, then the uniting control law presented in this work has strong theoretical guarantees. These theoretical guarantees can be improved in practice by empirically comparing the performance of fixed iteration limits to the uniting control law with terms $c_{p,0}$ and $~c_{p,1}.$ In practice, a user can empirically tune the $c_{p,0}$ and $~c_{p,1}$ terms by observing the plot of the Lyapunov function.

\subsection{Results Summary}
Section~\ref{sec:comparionOptimal} highlights for a practitioner there is a tradeoff between branch-and-bound iteration limits and quadratic programming iteration limits. In Example~\ref{prob:switching_thrusters} and Example~\ref{prob:min_thrust_MIQP}, branch-and-bound iteration limits result in smaller tracking error than quadratic programming iteration limits. This is to be expected because the branch-and-bound iterations enforce binary satisfaction from one time step to the MPC horizon N and quadratic programming iterations enforce constraint satisfaction at every time step. Section~\ref{sec:closer_look} highlights that for various iteration limits Algorithm~\ref{alg:bnb} reasonably tracks the optimal trajectory until near the origin where the numerics become unstable and Assumption~\ref{assum:asym_control}.2 may not hold in practice for sufficiently user-desired small iteration limits. Section~\ref{sec:compute_improvements} highlights there are problem dependent tradeoffs between $V_{obj},~V_{feas},$ and fixed iteration limits for average compute. Finally, Section~\ref{sec:tuning} discusses practical guidelines to tune $c_{p,0}$, $~c_{p,1},$ where numerical values for $c_{p,0}$, $~c_{p,1},$ are presented in Section~\ref{sec:results}.

\section{Conclusion}\label{sec:conclusion}
In this work, switching between iteration limited mixed-integer quadratic model predictive controllers (MIQP-MPC) is modeled as a hybrid system. Then stability results are derived for the theoretical model and implementable algorithms are developed to analyze the control and optimization closed loop system. Finally, simulations for the switching thruster and minimum thrust spacecraft rendezvous problems are used to verify the efficacy of the hybrid model for switching iteration limited MPC controllers. Future work will focus on application to hardware and theoretical analysis of Algorithm~\ref{alg:bnb} when modeled explicitly as a hybrid system.

\section{Appendix}\label{sec:appendix}
 The dynamics for the simulations are the Clohessy Wiltshire dynamics for spacecraft rendezvous where $\Delta T:=t-t_{0}$ is the sample time, $x_{k}$ technically denotes $x(\Delta T k)$ and the dynamics are
\begin{equation}
    A=\STM{t}{t_0}=\begin{bmatrix} \Phi_{\rm rr}(\Delta T) & \Phi_{\rm rv}(\Delta T) \\ \Phi_{\rm vr}(\Delta T) & \Phi_{\rm vv}(\Delta T) \end{bmatrix},
\end{equation}

where 

\begin{align*}
       &\Phi_{\rm rr}(t,t_0)= \begin{bmatrix} 4-3w_c & 0 & 0\\  6(w_s-w)&1&0\\0&0&w_c\end{bmatrix}, \\
        &\Phi_{\rm rv}(t,t_0)= \frac{1}{n_{L}}\begin{bmatrix} w_s & 2(1-w_c) & 0\\ 2(w_c-1)&4w_s-3w&0\\0&0&w_s\end{bmatrix},
        \\&\Phi_{\rm vr}(t,t_0)= n_{L}\begin{bmatrix} 3w_s & 0 & 0\\ 6(w_c-1)&0&0\\0&0&-w_s\end{bmatrix}, 
        \\&\Phi_{\rm vv}(t,t_0)= \begin{bmatrix} w_c & 2w_s & 0\\ -2w_s&4w_c-3&0\\0&0&w_c\end{bmatrix},        
\end{align*}
with $w=n_{L}(\Delta T)=n_{L}(t-t_0)$, $\cdot_c=\cos(\cdot)$, and $*_s=\sin(\cdot)$ for the sake of brevity where $n_{L}$ denotes the orbital rate. 
 
 The control matrices for Example~\ref{prob:switching_thrusters} models the two thruster control matrices as follows where
\begin{equation}
B_{1}=\begin{bmatrix}
    \Phi_{rv}(\Delta T) \\ \Phi_{vv}(\Delta T)
\end{bmatrix},
\end{equation}
is the control matrix when assuming a chemical impulsive thrust and
\begin{equation}\label{eq:elec_thrust}
B_{2}=\int_{0}^{k} \STM{\tau}{ t_0} \begin{bmatrix}
    \zeros{3} \\ \eye{3}/m_{s}
\end{bmatrix}d\tau,
\end{equation}
is the control matrix when assuming continuous electric thrust. The control matrix for Example~\ref{prob:min_thrust_MIQP} is the electric thruster model, $B_{2}.$ For this work, $m_{s}=100$kg and $n_{L}=1.13\times10^{-3}s^{-1}$.

\bibliographystyle{IEEEtran}
\bibliography{ref}

% \begin{IEEEbiography}[{\includegraphics[width=1in,clip,keepaspectratio]{figures/fina.jpg}}]{Luke Fina} received the M.S. degree in mechanical engineering from the University of Florida in 2023. He is currently working toward the Ph.D. degree in mechanical engineering degree at the University of Florida, Gainesville, FL USA.

% His research interest include mixed-integer optimization and applications for space craft systems with an emphasis on compute-constrained and distributed systems.
% \end{IEEEbiography}
% \vskip 0pt plus -1fil

% \begin{IEEEbiography}[{\includegraphics[height=1in,clip,keepaspectratio]{figures/petersen.png}}]{Christopher Petersen} received the Ph.D. degree in aerospace engineering from the University of Michigan in 2016.

% Dr. Christopher “Chrispy” Petersen is an Assistant Professor at the University of Florida in the Mechanical \& Aerospace Engineering Department as of Fall 2022.  He leads the Spacecraft Technology And Research (STAR) Laboratory, which is focused on real-time spacecraft guidance, navigation, control, and autonomy (GNCA).  Before Florida, he was at the Air Force Research Lab (AFRL) where he worked on 10+ satellite experiments, developing, deploying, and executing GNCA algorithms for ground and on-orbit use.  For his accomplishments, in 2021 he was awarded the AFRL Early Career Award.
% \end{IEEEbiography}

\end{document}